\documentclass[a4paper,11pt,twoside]{amsart}

\usepackage{graphicx}
\usepackage{amsfonts}
\usepackage{amsmath}
\usepackage{amsthm}
\usepackage{amssymb}
\usepackage[all]{xy}
\usepackage{hyperref}
\usepackage{aliascnt}

\newtheorem{thm}{Theorem}[section]

\newtheorem{thmInt}{Theorem}[section]

\newaliascnt{prop}{thm}
\newtheorem{prop}[prop]{Proposition}
\aliascntresetthe{prop}

\newaliascnt{lem}{thm}
\newtheorem{lem}[lem]{Lemma}
\aliascntresetthe{lem}

\newaliascnt{cor}{thm}
\newtheorem{cor}[cor]{Corollary}
\aliascntresetthe{cor}

\theoremstyle{definition}

\newaliascnt{definition}{thm}
\newtheorem{definition}[definition]{Definition}
\aliascntresetthe{definition}

\newaliascnt{remark}{thm}
\newtheorem{remark}[remark]{Remark}
\aliascntresetthe{remark}

\newaliascnt{ex}{thm}
\newtheorem{ex}[ex]{Example}
\aliascntresetthe{ex}

\numberwithin{equation}{section}

\newcommand{\ZZ}{\mathbb{Z}}

\newcommand{\iso}{\cong}

\newcommand{\farg}{-} 
\newcommand{\id}{\mathrm{id}}
\newcommand{\st}{:} 
\newcommand{\comp}{\circ} 
\newcommand{\mor}[1]{\xrightarrow{#1}}
\newcommand{\mono}{\hookrightarrow} 
\newcommand{\epi}{\twoheadrightarrow} 
\newcommand{\isomor}{\mor{\sim}} 
\newcommand{\rest}[1]{|_{#1}} 
\newcommand{\K}{\Bbbk} 

\newcommand{\Der}{\mathrm{Der}}
\newcommand{\coDer}{\mathrm{coDer}}
\newcommand{\md}{\omega} 
\newcommand{\cat}[1]{{\mathbf{#1}}} 

\newcommand{\Ob}[1]{\mathrm{Ob}(#1)} 
\newcommand{\dgMor}[1]{\underline{\mathrm{Mor}}(#1)} 
\newcommand{\pob}[1]{P(#1)} 
\newcommand{\fun}[1]{\mathsf{#1}} 
\newcommand{\Ho}[1]{\mathrm{Ho}(#1)} 
\newcommand{\red}[1]{\overline{#1}} 
\newcommand{\aug}[1]{#1^+} 
\newcommand{\B}{\mathrm{B}} 
\newcommand{\Ba}{\aug{\B}} 
\newcommand{\Bi}{\B_{\infty}} 
\newcommand{\Bia}{\aug{\Bi}} 
\newcommand{\coB}{\Omega} 
\newcommand{\coBa}{\aug{\coB}} 
\newcommand{\comult}[1]{\Delta_{#1}} 
\newcommand{\strmor}[1]{\eta_{#1}} 
\newcommand{\augmor}[1]{\epsilon_{#1}} 
\newcommand{\adjpair}[4]{#1\colon#2\rightleftarrows#3:\!#4}
\newcommand{\sh}[2][1]{#2[#1]} 
\newcommand{\Ta}[1]{\mathrm{T}(#1)} 
\newcommand{\Tc}[1]{\mathrm{T}^c(#1)} 
\newcommand{\Tar}[1]{\red{\mathrm{T}}(#1)} 
\newcommand{\Tcr}[1]{\red{\mathrm{T}}^c(#1)} 
\newcommand{\tc}{\tau} 
\newcommand{\utc}[1]{\tc_{#1}} 
\newcommand{\Vid}[1]{J_{#1}} 
\newcommand{\htpn}{\sim} 
\newcommand{\hison}{\approx} 
\newcommand{\quot}[1]{/#1} 
\newcommand{\loccl}[1]{[#1]} 
\newcommand{\m}[2]{m^{#1}_{#2}} 
\newcommand{\IHom}{\mathbb{R}\underline{Hom}} 
\newcommand{\gr}[1]{\mathrm{gr}(#1)} 

\newcommand{\cA}{\cat{A}}
\newcommand{\cB}{\cat{B}}
\newcommand{\cC}{\cat{C}}

\newcommand{\cV}{\cat{V}}
\newcommand{\dgMod}[1]{#1\text{-}\cat{dgMod}} 
\newcommand{\dgAlg}{\cat{dgAlg}} 
\newcommand{\dgAlgn}{\cat{dgAlg^n}} 
\newcommand{\dgcoAlgn}{\cat{dgcoAlg^n}} 
\newcommand{\Cat}{\cat{Cat}} 
\newcommand{\coCat}{\cat{coCat}} 
\newcommand{\Catn}{\cat{Cat^n}} 
\newcommand{\coCatn}{\cat{coCat^n}} 
\newcommand{\Cata}{\cat{Cat^a}} 
\newcommand{\coCata}{\cat{coCat^a}} 
\newcommand{\gcoCat}{\cat{gcoCat}} 
\newcommand{\gCatn}{\cat{gCat^n}} 
\newcommand{\gcoCatn}{\cat{gcoCat^n}} 
\newcommand{\gcoCata}{\cat{gcoCat^a}} 
\newcommand{\dgCat}{\cat{dgCat}} 
\newcommand{\dgcoCat}{\cat{dgcoCat}} 
\newcommand{\dgCatn}{\cat{dgCat^n}} 
\newcommand{\dgcoCatn}{\cat{dgcoCat^n}} 
\newcommand{\dgCata}{\cat{dgCat^a}} 
\newcommand{\dgCatc}{\cat{dgCat^c}} 
\newcommand{\dgcoCata}{\cat{dgcoCat^a}} 
\newcommand{\Qu}{\cat{Qu}} 
\newcommand{\Qua}{\cat{Qu^a}} 
\newcommand{\gQu}{\cat{gQu}} 
\newcommand{\gQua}{\cat{gQu^a}} 
\newcommand{\dgQu}{\cat{dgQu}} 
\newcommand{\dgQua}{\cat{dgQu^a}} 
\newcommand{\ACat}{\cat{A_\infty Cat}} 
\newcommand{\ACatn}{\cat{A_\infty Cat^n}} 
\newcommand{\ACata}{\cat{A_\infty Cat^a}} 
\newcommand{\ACatc}{\cat{A_\infty Cat^c}} 
\newcommand{\ACatcu}{\cat{A_\infty Cat^c_u}} 
\newcommand{\trivquiv}[1]{\cat{I}_{#1}} 
\newcommand{\trivcat}[1]{\cat{E}_{#1}} 
\newcommand{\Hqe}{\Ho{\dgCat}} 
\newcommand{\HoACat}{\Ho{\ACat}} 
\newcommand{\HoACatc}{\Ho{\ACatc}} 
\newcommand{\HoACatcu}{\Ho{\ACatcu}} 
\newcommand{\QACat}{\ACat\quot{\hison}} 
\newcommand{\QACatc}{\ACatc\quot{\hison}} 
\newcommand{\ACatdg}{\cat{A_\infty Cat_{dg}}} 
\newcommand{\HoACatdg}{\Ho{\ACatdg}} 
\newcommand{\QACatdg}{\ACatdg\quot{\hison}} 

\newcommand{\FunA}{\cat{Fun}_{\ACat}}

\newcommand{\FunAc}{\cat{Fun}_{\ACatc}}

\newcommand{\FunAn}{\cat{Fun}_{\ACatn}}

\newcommand{\fF}{\fun{F}}
\newcommand{\fG}{\fun{G}}
\newcommand{\fH}{\fun{H}}
\newcommand{\fI}{\fun{I}}
\newcommand{\fJ}{\fun{J}}
\newcommand{\fK}{\fun{K}}

\newcommand{\fM}{\fun{M}}
\newcommand{\fN}{\fun{N}}
\newcommand{\fQ}{\fun{Q}}
\newcommand{\fR}{\fun{R}}
\newcommand{\fS}{\fun{S}}
\newcommand{\fT}{\fun{T}}
\newcommand{\fU}{\fun{U}}
\newcommand{\fdgA}{\fI} 
\newcommand{\fdgAn}{\fdgA^{\fun{n}}} 
\newcommand{\VdB}{\fU} 
\newcommand{\VdBn}{\VdB^{\fun{n}}} 
\newcommand{\fucu}{\fJ} 
\newcommand{\fcuu}{\fK} 
\newcommand{\fhcuu}{\fcuu'} 
\newcommand{\finc}{\fI} 
\newcommand{\fpro}{\fQ} 
\newcommand{\dgmf}[1]{\tilde{#1}} 
\newcommand{\rcf}[1]{\fR_{#1}} 

\newcommand{\ncb}[1][]{\alpha_{#1}} 
\newcommand{\nbc}[1][]{\beta_{#1}} 
\newcommand{\ncbi}[1][]{\gamma_{#1}} 
\newcommand{\ntV}[1][]{\rho_{#1}} 
\newcommand{\nfV}[1][]{\sigma_{#1}} 
\newcommand{\nhtV}[1][]{\Ho{\rho}_{#1}} 
\newcommand{\nhfV}[1][]{\Ho{\sigma}_{#1}} 
\newcommand{\Nat}{\mathrm{Nat}} 

\begin{document}

	\title[Localizations of the category of $A_\infty$ categories and internal Homs]{Localizations of the category of $A_\infty$ categories and internal Homs}

	\author[A.~Canonaco]{Alberto Canonaco}
	\address{\parbox{0.9\textwidth}{Universit{\`a} degli Studi di Pavia\\[1pt]
			Dipartimento di Matematica ``F. Casorati''\\
			Via Ferrata 5, 27100 Pavia, Italy
			\vspace{4pt}
	}}
	\email{alberto.canonaco@unipv.it \vspace{0,1cm}}

    \author[M.~Ornaghi]{Mattia Ornaghi}
   \address{\parbox{0.9\textwidth}{Ben Gurion University\\[1pt]
   Department of Mathematics\\[1pt]
   Be'er Sheva 84105, Israel
   		\vspace{4pt}
   		}}
   \email{mattia12.ornaghi@gmail.com}
   \urladdr{\url{https://sites.google.com/view/mattiaornaghi}\vspace{0,1cm}}
    
    \author[P.~Stellari]{Paolo Stellari}
    \address{\parbox{0.9\textwidth}{Universit\`a degli Studi di Milano\\[1pt]
    		Dipartimento di Matematica ``F.~Enriques''\\[1pt]
    		Via Cesare Saldini 50, 20133 Milano, Italy
    		\vspace{4pt}
    }}
    \email{paolo.stellari@unimi.it}
    \urladdr{\url{http://sites.unimi.it/stellari}}
	
	\thanks{A.~C.~ was partially supported by the national research project
	  ``Moduli spaces and Lie theory'' (PRIN 2015).
	P.~S.~ was partially supported by the ERC Consolidator Grant ERC-2017-CoG-771507-StabCondEn and by the research projects FIRB 2012 ``Moduli Spaces and Their Applications'' and PRIN 2015 ``Geometria delle variet\`a proiettive''.}



\begin{abstract}
We prove that the localizations of the categories of dg categories, of cohomologically unital and strictly unital $A_\infty$ categories with respect to the corresponding classes of quasi-equivalences are all equivalent. Moreover we show that the last two localizations are equivalent to the corresponding quotients by the relation of being isomorphic in the cohomology of the $A_\infty$ category of $A_\infty$ functors. As an application we give a complete proof of a claim by Kontsevich stating that the category of internal Homs for two dg categories can be described as the category of strictly unital $A_\infty$ functors between them.
\end{abstract}

\maketitle

\setcounter{tocdepth}{1}
\tableofcontents

\section*{Introduction}

The category $\dgCat$ of differential graded (dg from now on) categories defined over a field $\K$ is widely studied. Roughly speaking, a dg category is a category whose space of morphism is actually a complex.  Due to the work of Tabuada \cite{Ta}, $\dgCat$ has a model structure which allows to give a very nice description of the localization $\Hqe$ of such a category by the class of quasi-equivalences. Actually, the category $\Hqe$ provides the correct framework to look for dg enhancements of triangulated categories and dg lifts of exact functors between triangulated categories (see \cite{CSSurvey} for an overview on the subject).

Let us look closely at triangulated categories of algebro-geometric nature: the unbounded derived category of quasi-coherent sheaves, the bounded derived category of coherent sheaves and the category of perfect complexes on an algebraic stack. They all possess natural dg enhancements. But other geometric categories arising from symplectic geometry come with a slightly more general kind of enhancements: the $A_\infty$ ones. This is the case of the Fukaya category which is related to lagrangian submanifolds of a smooth symplectic manifold. The interplay between triangulated categories of algebro-geometric or symplectic type is highly non-trivial and at the very forefront of modern geometry, as predicted by the celebrated homological version of the Mirror Symmetry Conjecture due to Kontsevich \cite{Ko}.

This pushes the attention to $A_\infty$ categories and functors which, in this paper, will always be meant to be defined over a field. Strictly speaking an $A_\infty$ category is not a category as the composition is associative only up to higher multiplication maps, contrary to the case of dg categories.

In addition to this, $A_\infty$ categories have various possible incarnations. Indeed, one can talk about strictly unital or cohomologically unital $A_\infty$ categories. While the former have identity morphisms and are thus closer to being categories, only the cohomology categories of the latter have identities and are thus categories in a proper sense. To fix the notation, we denote by $\ACat$ (resp.\ $\ACatc$) the category of strictly (resp.\ cohomologically) unital $A_\infty$ categories, whose morphisms are strictly (resp.\ cohomologically) unital $A_\infty$ functors.

It is important to keep in mind that we need to distinguish between these different types of categories not for pure abstraction but due to geometry. Indeed, while the dg categories enhancing algebro-geometric categories are strictly unital, the Fukaya category is by no means strictly unital in a natural way but it is cohomologically unital.

Thus, one can either take the localization $\HoACat$ of $\ACat$ with respect to strictly unital quasi-equivalences or the localization $\HoACatc$ of $\ACatc$ with respect to quasi-equivalences. We can go further and consider the quotients of $\ACatc$ and $\ACat$ by another crucial equivalence relation. Indeed, if $\cA_1$ and $\cA_2$ are cohomologically or strictly unital $A_\infty$ categories one can form the $A_\infty$ category $\FunAc(\cA_1,\cA_2)$ which will be carefully defined in \autoref{subsect:Ainftyfunctors}. Roughly, its objects are the cohomologically unital $A_\infty$ functors between $\cA_1$ and $\cA_2$.  We say that two cohomologically unital or strictly unital $A_\infty$ functors $\fF_1,\fF_2\colon\cA_1\to\cA_2$ are equivalent $\fF_1\hison\fF_2$ if they are isomorphic in the $0$-th cohomology of $\FunAc(\cA_1,\cA_2)$. Thus we can take the quotients $\QACat$ and $\QACatc$ of $\ACat$ and $\ACatc$ with respect to this relation.

Our first main result is then the following.

\begin{thmInt}\label{thm:localizations}
The  faithful but not full natural inclusions $\dgCat\mono\ACat\mono\ACatc$ induce equivalences
\[
\Hqe\iso\HoACat\iso\HoACatc.
\]
Moreover, these categories are equivalent to $\QACatc$ and $\QACat$.
\end{thmInt}

This result is certainly expected by experts but we could not find a proof in the existing literature which may work in the generality mentioned above. Actually, \autoref{thm:localizations} has various interesting applications. The first easy one is that a triangulated category has a unique dg enhancement if and only if it has unique strictly or cohomologically unital $A_\infty$ ones. The uniqueness of dg enhancements for the algebro-geometric categories mentioned above was conjectured by Bondal, Larsen and Lunts \cite{BLL} for smooth projective varieties. Such a conjecture was proved in a much more general setting by Lunts and Orlov \cite{LO}. In \cite{CS6}, these results were further extended and \autoref{thm:localizations} implies that the same results hold for $A_\infty$ enhancements.

\medskip

Let us now move to a much more elaborate application. It is well-known and not difficult to prove that $\dgCat$ is a closed symmetric monoidal category with respect to the tensor product of dg categories. The important feature is that the tensorization has a right adjoint given by the dg category of dg functors between the two dg categories which we are given.

It is a much deeper and recent result that $\Hqe$ is a closed symmetric monoidal category, again with respect to the tensor product of dg categories and with right adjoint to the tensorization given by the dg category of \emph{internal Homs} between the corresponding dg categories. More precisely, let $\cA_1$, $\cA_2$ and $\cA_3$ be three dg categories. Then the category $\IHom(\cA_2,\cA_3)$ of internal Homs between $\cA_2$ and $\cA_3$, is the unique (up to isomorphism in $\Hqe$) dg category yielding a natural bijection
\[
\xymatrix{
	\Hqe(\cA_1\otimes\cA_2,\cA_3) \ar@{<->}[rr]^-{1:1} & & \Hqe(\cA_1,\IHom(\cA_2,\cA_3)).
}\label{eqn:defIntHoms}\tag{IH}
\]
This result was first proven by To\"en \cite{To} and later reproven in \cite{CS} in a much simpler way.

The point is that, a while before \cite{To}, Kontsevich had a very bright vision of how to prove the existence of the category of internal Homs and of its explicit description. This can be summarized as follows:

\medskip

\noindent{\bf Claim} (Kontsevich){\bf .} \emph{If $\cA_1$ and $\cA_2$ are dg categories, then the dg category $\FunA(\cA_1,\cA_2)$, whose objects are strictly unital $A_\infty$ functors, is the category of internal Homs between $\cA_1$ and $\cA_2$.}

\medskip

There are several pointers in the literature to the above claim (see \cite{Dr,K2,To}). But, quite surprisingly, no correct and complete proof of this very nice statement seems to be available. We will discuss this problem a bit later.

Our second main result fills this gap, by using \autoref{thm:localizations} in a crucial way.

\begin{thmInt}\label{thm:internalHoms}
Let $\cA_1$, $\cA_2$ and $\cA_3$ be three dg categories. Then there exists a natural bijection
\begin{equation*}\label{eqn:ultima}
\xymatrix{
\Hqe(\cA_1\otimes\cA_2,\cA_3) \ar@{<->}[rr]^-{1:1} & & \Hqe(\cA_1,\FunA(\cA_2,\cA_3))
}
\end{equation*}
proving that the symmetric monoidal category $\Hqe$ is closed.	
\end{thmInt}

As a consequence, $\IHom(\cA_2,\cA_3)$ is isomorphic to $\FunA(\cA_2,\cA_3)$ in $\Hqe$. The advantage of using this description of the category $\IHom(\cA_2,\cA_3)$ is that morphisms in $\Hqe$ can now be described as equivalence classes of strictly unital $A_\infty$ functors and not as roofs of dg functors. We hope that this may be used in some future work to give a simpler proof of the fact that an exact functor between the bounded derived categories of coherent sheaves on smooth projective schemes over a field can be lifted to a morphism in $\Hqe$ if and only if it is of Fourier--Mukai type (see \cite{LS,To}).

\subsection*{Related work}

A version of \autoref{thm:localizations} was proved by Lef\`evre-Hasegawa in \cite{LH} for (non-unital) $A_\infty$ algebras. Actually his result extends to the category of (non-unital) $A_\infty$ categories with the same set of objects, but this is clearly not sufficiently general for our purposes.

Lef\`evre-Hasegawa's proof is based on the observation that, even though the category of $A_\infty$ algebras does not have a model structure, it can be endowed with a `degenerate' model structure without arbitrary limits and colimits. This is related to a true model structure on the category of dg coalgebras which was obtained in \cite{LH} in analogy with a result of Hinich \cite{H} for dg algebras.

When we pass to $\ACat$ and $\ACatc$, the situation is very similar. Indeed, both categories do not have a model structure with arbitrary limits and colimits (see \autoref{subsec:nomodel}). Clearly, one could follow the same strategy as in \cite{LH} to prove \autoref{thm:localizations}. The issue is that Lef\`evre-Hasegawa's result mentioned above is not at hand for dg cocategories. Actually, it seems a non-trivial problem in itself to show that the category of dg cocategories admits an interesting model structure.
To avoid these delicate problems, we proceed by constructing explicit equivalences between the given categories. This has the advantage of providing a very handy control on the various localizations.

As for \autoref{thm:internalHoms}, to the best of our knowledge, the only other paper dealing with Kontsevich's claim is \cite{Fao}, where it appears as Theorem 1.2. Such a result is used in the same paper to prove some interesting properties concerning the mapping spaces of dg categories and the Hochschild cohomology of $A_\infty$ categories (see Theorems 1.3 and Theorem 1.4 in \cite{Fao}). Unfortunately, after a careful analysis, the proof in \cite{Fao} turned out to be a bit too rough and incorrect at some steps, probably due to a wrong use of the notion of augmentation and reduction. More precisely, Definitions A.14 and A.16 in \cite{Fao} look inaccurate to us and, as a consequence, Lemmas A.20 and A.21 of the same paper are clearly wrong. These lemmas are used in Remark A.26 of \cite{Fao} where it is erroneously claimed that a certain functor $\gamma_D$ is a quasi-equivalence. All these results are correct in the augmented but not in the strictly unital case. In the latter setting, the appropriate statement is our \autoref{dgAadj}. Even if we replace Remark A.26 with it, the rest of the proof of Kontsevich's claim in Section 2 of \cite{Fao} is quite hard to follow. The reason being that the author refers (in a slightly vague way) to \cite{LH} without making the necessary upgrade of the results in \cite{LH} from the augmented to the strictly unital setting. We are aware that Faonte is currently working on a revision of \cite{Fao} after we sent him a preliminary version of this paper.

It should also be noted that our proof and the approach by Faonte are different in spirit. Indeed, Faonte proceeds by showing that, given two dg categories $\cA_1$ and $\cA_2$, the dg category $\IHom(\cA_2,\cA_2)$ described by To\"en (see \cite{To,CS}) is quasi-equivalent to $\FunA(\cA_1,\cA_2)$. Our proof consists in showing the existence of the category of internal Homs from scratch by directly proving that $\FunA(\cA_1,\cA_2)$ satisfies the defining property \eqref{eqn:defIntHoms} of the dg category of internal Homs. In other words, we provide yet another proof of the main result in \cite{To} for the special case when the dg categories are defined over a field.

\subsection*{Plan of the paper}

\autoref{sec:introAinfty} is a rather short introduction to $A_\infty$ categories and $A_\infty$ functors providing precise definitions for all the notions mentioned above. We compare the approaches in \cite{BLM} and \cite{Sei} and introduce the notion of $A_\infty$ multifunctor (see \autoref{subsect:Ainftyfunctors}). We recall the bar and cobar constructions in \autoref{subsec:barcobar} as these notions will be fundamental all along the paper. The last argument treated in this part of the paper is the lack of limits (whence of a model structure) for the categories of strictly and cohomologically unital $A_\infty$ categories (see \autoref{subsec:nomodel}).

\autoref{sect:proofthm1} is completely devoted to the proof of \autoref{thm:localizations}. In particular, the latter result is the combination of \autoref{thm:1}, \autoref{thm:2}, \autoref{ucueq} and \autoref{thm:3}. Finally, the proof of \autoref{thm:internalHoms} is contained in \autoref{sect:internalHoms}, where the use of $A_\infty$ multifunctors is crucial.

\subsection*{Notation and conventions}

We assume that a universe containing an infinite set is fixed. Throughout the paper, we will simply call sets the members of this universe. In general the collection of objects of a category need not be a set: we will always specify if we are requiring this extra condition.

We work over a field $\K$. We will always assume that the collection of objects in a $\K$-linear category is a set. The category (whose collection of objects is not a set) of $\K$-linear categories and $\K$-linear functors will be denoted by $\Cat$. We will also use the more general notions of non-unital $\K$-linear categories and functors, which form a category $\Catn$.

The shift by an integer $n$ of a graded or dg object $M$ will be denoted by $\sh[n]{M}$.

\section{Preliminaries on $A_\infty$ categories and functors}\label{sec:introAinfty}

In this section we recall some basic facts about $A_\infty$ categories and (multi)functors. The key results for us are those pointing to the adjunction between the bar and cobar constructions. They form the bulk of this section.

\subsection{$A_\infty$ categories and $A_\infty$ functors}\label{subsect:Ainftycategories}

Let us start with the basic definition of $A_\infty$ categories and its variants depending on the presence of units, at different levels. We will not follow the sign convention in \cite{Sei} but the equivalent one in \cite{LH}. This is motivated by the fact that the latter is automatically compatible with the bar and cobar constructions that we will discuss later in this section.

\begin{definition}\label{def:nonunitalAinftycat}
A \emph{non-unital $A_\infty$ category} $\cA$ consists of a set of objects $\Ob{\cA}$ and, for any $A_1,A_2\in\Ob{\cA}$, a $\ZZ$-graded $\K$-vector space $\cA(A_1,A_2)$ and, for all $i\ge1$, $\K$-linear maps
\[
\m{i}{\cA}\colon\cA(A_{i-1},A_i)\otimes\cdots\otimes\cA(A_{0},A_1)\to\cA(A_0,A_i)
\]
of degree $2-i$. The maps must satisfy the \emph{$A_\infty$ associativity relations}
\begin{equation}\label{eqn:catrel}
\sum_{\substack{0\leq j< n\\1\leq k\leq n-j}}(-1)^{jk+n-j-k}\m{n-k+1}{\cA}(\id^{\otimes j}\m{k}{\cA}\otimes\id^{\otimes n-j-k})=0,
\end{equation}
for every $n\geq 1$.
\end{definition}

The above relation when $n=1$ implies that the pair $(\cA(A_1,A_2), \m{1}{\cA})$ is a complex, for every $A_1,A_2\in\Ob{\cA}$. On the other hand, for $n=2$ we have that $\m{1}{\cA}$ satisfies the (graded) Leibniz rule with respect to the composition defined by $\m{2}{\cA}$. Moreover, for $n=3$ we have that $\m{2}{\cA}$ is associative up to a homotopy defined by $\m{3}{\cA}$. 

In particular, we can then take cohomologies of $\cA(A_1,A_2)$ and define the (graded, non-unital) \emph{cohomology category} $H(\cA)$ of $\cA$ such that $\Ob{H(\cA)}=\Ob{\cA}$ and
\[
H(\cA)(A_1,A_2)=\bigoplus_i H^i(\cA(A_1,A_2)).
\]
The latter graded vector spaces come with the natural induced associative composition
\[
[f_2]\comp[f_1]:=[\m{2}{\cA}(f_2,f_1)].
\]

Even though we call $H(\cA)$ a category, it may not be a category in a strict sense, as $H(\cA)$ may lack identities. For this reason, we introduce the following special classes of $A_\infty$ categories.

\begin{definition}\label{def:cunitalAinftycat}
A \emph{cohomologically unital $A_\infty$ category} is a non-unital $A_\infty$ category $\cA$ such that $H(\cA)$ is a category (i.e.\ $H(\cA)$ is unital).
\end{definition}

The above definition can be made stricter.

\begin{definition}\label{def:unitalAinftycat}
A \emph{strictly unital $A_\infty$ category} is a non-unital $A_\infty$ category $\cA$ such that, for all $A\in\Ob{\cA}$, there exists a degree $0$ morphisms $\id_A\in\cA(A,A)$ such that
\begin{align*}
&\m{1}{\cA}(\id_A)=0,\\
&\m{2}{\cA}(f,\id_A)=\m{2}{\cA}(\id_A,f)=f,\\
&\m{i}{\cA}(f_{i-1},\ldots,f_k,\id_A,f_{k-1},\ldots,f_1)=0,
\end{align*}
for all morphisms $f,f_1,\ldots,f_{i-1}$ and for all $i>2$ and $1\leq k\leq i$.
\end{definition}

Clearly, a strictly unital $A_\infty$ category is cohomologically unital. A non-unital $A_\infty$ category $\cA$ such that $\m{i}{\cA}=0$, for all $i>2$, is called a \emph{non-unital dg category}. In complete analogy, we get \emph{cohomologically unital dg categories} and \emph{strictly unital dg categories}. In accordance to the existing literature, strictly unital dg categories will be simply referred to as \emph{dg categories}.

When the corresponding (dg or $A_\infty$) categories have only one object, then, for obvious reasons, we will talk about (dg or $A_\infty$) algebras.

\begin{ex}\label{trivcat}
For every set $O$ we will denote by $\trivcat{O}$ the dg category whose set of objects is $O$ and whose morphisms are given by
\[
\trivcat{O}(A_1,A_2)=\begin{cases}
\K\,\id_{A_1} & \text{if $A_1=A_2$} \\
0 & \text{otherwise,}\end{cases}
\]
with zero differential and obvious composition.
\end{ex}

\begin{ex}\label{ex:tensor}
Given two (non-unital, cohomologically unital or strictly unital) dg categories $\cA_1$ and $\cA_2$ we can define a (non-unital, cohomologically unital or strictly unital) dg category $\cA_1\otimes\cA_2$ which is the tensor product of $\cA_1$ and $\cA_2$. Its objects are the pairs $(A_1,A_2)$ with $A_i\in\Ob{\cA_i}$ while $\cA_1\otimes\cA_2((A_1,A_2),(B_1,B_2))=\cA_1(A_1,B_1)\otimes\cA_2(A_2,B_3)$. If $\cA_1$ and $\cA_2$ are $A_\infty$ categories, then there is no well behaved and simple notion of tensor product as in the dg case. It would be very useful to fill such a lack.
\end{ex}

\begin{definition}\label{def:Ainftyfunctors}
A \emph{non-unital $A_\infty$ functor} $\fF\colon\cA_1\to\cA_2$ between two non-unital $A_\infty$ categories $\cA_1$ and $\cA_2$ is a collection $\fF=\{\fF^i\}_{i\geq 0}$ such that $\fF^0\colon\Ob{\cA_1}\to\Ob{\cA_2}$ is a map of sets and
\[
\fF^i\colon\cA_1(A_{i-1},A_i)\otimes\cdots\otimes\cA_1(A_{0},A_1)\to\cA_2(\fF^0(A_0),\fF^0(A_i))
\]
are degree $1-i$ maps of graded vectors spaces, for all $A_0,\ldots,A_i\in\Ob{\cA_1}$, satisfying the following relations
\begin{multline}\label{eqn:funrel}
\sum_{\substack{0\leq j< n\\ 1\leq k\leq n-j}}(-1)^{jk+n-j-k}\fF^{n-k+1}(\id^{\otimes j}\otimes\m{k}{\cA_1}\otimes\id^{\otimes n-j-k})\\
=\sum_{\substack{1\leq r\leq n\\s_1+\cdots+s_r=n}}(-1)^{\sum_{2\leq u\leq r}\left((1-s_u)\sum_{1\leq v\leq u}s_v\right)}\m{r}{\cA_2}(\fF^{s_r}\otimes\ldots\otimes\fF^{s_1}),
\end{multline}
for every $n\ge1$. A non-unital $A_\infty$ functor $\fF$ is \emph{strict} if $\fF^i=0$ for every $i>1$.
\end{definition}

The above relation when $n=1$ implies that $\fF^1$ commutes with the differentials $\m{1}{\cA_i}$. On the other hand, for $n=2$ we see that $\fF^1$ preserves the compositions $\m{2}{\cA_i}$, up to a homotopy defined by $\fF^2$. It follows that $\fF^1$ induces a non-unital graded functor
\[
H(\fF)\colon H(\cA_1)\to H(\cA_2).
\]
A non-unital $A_\infty$ functor $\fF$ is a \emph{quasi-isomorphism} if $H(\fF)$ is an isomorphism.

\begin{definition}\label{def:cunitAinftyfunctors}
If $\cA_1$ and $\cA_2$ are cohomologically unital $A_\infty$ categories, a non-unital $A_\infty$ functor $\fF\colon\cA_1\to\cA_2$ is \emph{cohomologically unital} if $H(\fF)$ is unital.
\end{definition}

A cohomologically unital $A_\infty$ functor $\fF$ is a \emph{quasi-equivalence} if $H(\fF)$ is an equivalence.

\begin{definition}\label{def:sunitAinftyfunctors}
If $\cA_1$ and $\cA_2$ are strictly unital $A_\infty$ categories, a non-unital $A_\infty$ functor $\fF\colon\cA_1\to\cA_2$ is \emph{strictly unital} if $\fF^1(\id_A)=\id_{\fF^0(A)}$ and $\fF^i(f_{i-1},\ldots,\id_A,\ldots,f_1)=0$, for all $A\in\Ob{\cA_1}$ and all morphisms $f_1,\ldots,f_{i-1}$.
\end{definition}

Non-unital, cohomologically unital and strictly unital $A_\infty$ functors can be composed in an explicit way that will be made clear later (see \autoref{rmk:Bifun}) producing the same type of $A_\infty$ functors. We denote by $\ACatn$ the category whose objects are non-unital $A_\infty$ categories and whose morphisms are non-unital $A_\infty$ functors. Analogously, we get the categories $\ACatc$ and $\ACat$ whose objects and morphisms are, respectively, cohomologically unital and strictly unital $A_\infty$ categories and functors. Similarly, we denote by $\dgCatn$, $\dgCatc$ and $\dgCat$ the categories whose objects are non-unital, cohomologically unital and strictly unital dg categories with the corresponding strict $A_\infty$ functors (called \emph{dg functors}).

\begin{ex}
A map of sets $f\colon O\to O'$ defines a unique dg functor $\trivcat{f}\colon\trivcat{O}\to\trivcat{O'}$ which coincides with $f$ on objects.
\end{ex}

\begin{definition}
A strictly unital $A_\infty$ category $\cA$ is \emph{augmented} if it is endowed with a strict and strictly unital $A_\infty$ functor $\augmor{\cA}\colon\cA\to\trivcat{\Ob{\cA}}$ which is the identity on objects.
\end{definition}

If $\cA$ is an augmented $A_\infty$ category, its \emph{reduction} is the non-unital $A_\infty$ category $\red{\cA}$ such that $\Ob{\red{\cA}}=\Ob{\cA}$ and $\red{\cA}(A_1,A_2)=\ker(\augmor{\cA}\colon\cA(A_1,A_2)\to\trivcat{\Ob{\cA}}(A_1,A_2))$ (for every $A_1,A_2\in\cA$), with $\m{i}{\red{\cA}}$ induced from $\m{i}{\cA}$ by restriction.

Given a non-unital $A_\infty$ category $\cA$, its \emph{augmentation} is the augmented $A_\infty$ category $\aug{\cA}$ such that $\Ob{\aug{\cA}}=\Ob{\cA}$ and
\[
\aug{\cA}(A_1,A_2)=\begin{cases}
\cA(A_1,A_2) & \text{if $A_1\neq A_2$} \\
\cA(A_1,A_1)\oplus\K\,1_{A_1} & \text{otherwise,}\end{cases}
\]
with $\m{i}{\aug{\cA}}$ the unique extension of $\m{i}{\cA}$ such that the additional morphisms $1_A$ is the unit of $A$ in $\aug{\cA}$, for every $A\in\cA$. Let us stress that, if $\cA$ is strictly unital, the unit in $\cA$ is denoted by $\id_A$ while the one on $\aug{\cA}$ is $1_A$, for every $A\in\cA$. 

Augmented $A_\infty$ categories form a category $\ACata$ with
\[
\ACata(\cA_1,\cA_2):=\{\fF\in\ACat(\cA_1,\cA_2)\st\augmor{A_2}\comp\fF=\trivcat{\fF^0}\comp\augmor{A_1}\}.
\]
It is easy to see that the maps $\cA\mapsto\red{\cA}$ and $\cA\mapsto\aug{\cA}$ extend to functors $\ACata\to\ACatn$ and $\ACatn\to\ACata$, which are quasi-inverse equivalences of categories. Denoting by $\dgCata$ the subcategory of $\ACata$ whose objects are (augmented) dg categories and whose morphisms are (compatible) dg functors, these equivalences clearly restrict to quasi-inverse equivalences between $\dgCata$ and $\dgCatn$.

\subsection{Bar construction and $A_\infty$ multifunctors}\label{subsec:barcobar}

First we need to recall basic facts about (dg) quivers and (dg) cocategories. We refer to \cite{BLM,K1} for extensive presentations.

\begin{definition}
A ($\K$-linear) \emph{quiver} $\cV$ consists of a set of objects $\Ob{\cV}$ and, for every $X,Y\in\cV$, of a $\K$-vector space $\cV(X,Y)$. A morphism of quivers $\fF\colon\cV\to\cV'$ is given by maps $\fF\colon\Ob{\cV}\to\Ob{\cV'}$ and (for every $X,Y\in\cV$) $\fF=\fF_{X,Y}\colon\cV(X,Y)\to\cV'(\fF(X),\fF(Y))$.
\end{definition}

\begin{ex}
For every set $O$ we will denote by $\trivquiv{O}$ the quiver defined as $\trivcat{O}$ in \autoref{trivcat} (forgetting differential and composition).
\end{ex}

Quivers and morphisms of quivers (with the obvious composition) clearly form a category $\Qu$.

\begin{definition}
A quiver $\cV$ is \emph{augmented} if it is endowed with two (structure) morphisms of quivers $\trivquiv{\Ob{\cV}}\to\cV\to\trivquiv{\Ob{\cV}}$ which are the identities on objects and whose composition is the identity. A morphism of augmented quivers is a morphism of quivers which is compatible with the structure morphisms.
\end{definition}

Again, augmented quivers and augmented morphisms form a category $\Qua$. Similarly as before, there are natural reduction and augmentation functors which give quasi-inverse equivalences between $\Qua$ and $\Qu$.

Replacing (in the above definitions) vector spaces with graded (respectively, dg) vector spaces, one gets the notions of (augmented) graded (respectively dg) quivers. The corresponding categories will be denoted by $\gQu$ and $\gQua$ (respectively, $\dgQu$ and $\dgQua$).

\begin{definition}
A ($\K$-linear) \emph{non-unital cocategory} is a quiver $\cC$ endowed with $\K$-linear maps
\[
\comult{\cC}\colon\cC(X,Y)\to\bigoplus_{Z\in\cC}\cC(Z,Y)\otimes\cC(X,Z)
\]
(for every $X,Y\in\cC$) satisfying the natural coassociativity condition.

A non-unital cocategory $\cC$ is \emph{cocomplete} if each morphism of $\cC$ is in the kernel of a sufficiently high iterate of $\comult{\cC}$.
\end{definition}

\begin{ex}
For a quiver $\cV$ the (reduced) tensor cocategory $\Tcr{\cV}$ has $\Ob{\Tcr{\cV}}:=\Ob{\cV}$,
\[
\Tcr{\cV}(X,Y):=\cV(X,Y)\bigoplus_{n>0}\bigoplus_{Z_1,\dots,Z_n\in\cV}\cV(Z_n,Y)\otimes\cV(Z_{n-1},Z_n)\otimes\cdots\otimes\cV(Z_1,Z_2)\otimes\cV(X,Z_1)
\]
(for every $X,Y\in\cV$) and
\[
\comult{\Tcr{\cV}}(f_n,\dots,f_1):=\sum_{i=1}^{n-1}(f_n,\dots,f_{i+1})\otimes(f_i,\dots,f_1)
\]
(for every morphism $(f_n,\dots,f_1)$ of $\Tcr{\cV}$). It is easy to see that $\Tcr{\cV}$ is a cocomplete non-unital cocategory. Dually, we denote by $\Tar{\cV}$ the usual (reduced) tensor category over $\cV$.
\end{ex}

\begin{definition}
A \emph{non-unital cofunctor} $\fF\colon\cC_1\to\cC_2$ between two non-unital cocategories $\cC_1$ and $\cC_2$ is a morphism of quivers such that $\comult{\cC_2}\comp\fF=(\fF\otimes\fF)\comp\comult{\cC_1}$.
\end{definition}

We will denote by $\coCatn$ the category whose objects are non-unital cocomplete cocategories and whose morphisms are non-unital cofunctors. It is not difficult to prove the following result (see \cite[Lemma 5.2]{K1} for the augmented version of the first part or \cite[Lemma 1.1.2.2]{LH} for the case of (co)algebras).

\begin{prop}\label{quivadj}
The forgetful functor $\coCatn\to\Qu$ has right adjoint defined on objects by $\cV\mapsto\Tcr{\cV}$. Dually, the forgetful functor $\Catn\to\Qu$ has left adjoint defined on objects by $\cV\mapsto\Tar{\cV}$.
\end{prop}

\begin{definition}
Let $\fF_1,\fF_2\in\coCatn(\cC_1,\cC_2)$. A \emph{$(\fF_1,\fF_2)$-coderivation} is a collection of $\K$-linear maps $D\colon\cC_1(C_1,C_2)\to\cC_2(\fF_1(C_1),\fF_2(C_2))$, for every pair of objects $C_1,C_2\in\cC_1$, such that the relation
\[
\comult{\cC_2}\comp D=(\fF_2\otimes D+D\otimes\fF_1)\comp\comult{\cC_1}
\]
holds true.
\end{definition}

We will denote by $\coDer(\fF_1,\fF_2)$ the $\K$-module of $(\fF_1,\fF_2)$-coderivations. Dually, if $\fF_1,\fF_2\in\Catn(\cA_1,\cA_2)$, there is well known notion of $(\fF_1,\fF_2)$-derivation and we will denote by $\Der(\fF_1,\fF_2)$ the $\K$-module of $(\fF_1,\fF_2)$-derivations.

We will need the following result, which is not difficult to prove and  well known to experts (see Lemma 1.1.2.1 and 1.1.2.2 in \cite{LH} for the case of (co)algebras).

\begin{prop}\label{cofree}
Given $\cV\in\Qu$ and $\fF_1,\fF_2\in\coCatn(\cC,\Tcr{\cV})$, there is a natural isomorphism between $\coDer(\fF_1,\fF_2)$ and the $\K$-module consisting of collections of $\K$-linear maps
\[
\cC(C_1,C_2)\to\cV(\fF'_1(C_1),\fF'_2(C_2))
\]
(for every $C_1,C_2\in\cC$), where $\fF'_i\in\Qu(\cC,\cV)$ corresponds to $\fF_i$ under the adjunction of \autoref{quivadj}. Dually, given $\cV\in\Qu$ and $\fF_1,\fF_2\in\Catn(\Tar{\cV},\cA)$, there is a natural isomorphism between $\Der(\fF_1,\fF_2)$ and the $\K$-module consisting of collections of $\K$-linear maps
\[
\cV(V_1,V_2)\to\cA(\fF'_1(V_1),\fF'_2(V_2))
\]
(for every $V_1,V_2\in\cV$), where $\fF'_i\in\Qu(\cV,\cA)$ corresponds to $\fF_i$ under the adjunction of \autoref{quivadj}.
\end{prop}

There are also natural notions of unital and augmented cocategories; the corresponding categories will be denoted by $\coCat$ and $\coCata$. As usual, there are reduction and augmentation functors which give quasi-inverse equivalences between $\coCata$ and $\coCatn$. Dually, in the case of categories one can define the category $\Cata$ of ($\K$-linear) augmented categories, which is equivalent to $\Catn$. We also define, for every $\cV\in\Qu$, the tensor cocategory $\Tc{\cV}$ as $\aug{\Tcr{\cV}}\in\coCata$ and the tensor category $\Ta{\cV}$ as $\aug{\Tar{\cV}}\in\Cata$.

Replacing (in the above definitions) vector spaces with graded (respectively, dg) vector spaces, one gets the various notions of graded (respectively dg) cocategories. The corresponding categories will be denoted by $\gcoCatn$, $\gcoCat$ and $\gcoCata$ (respectively, $\dgcoCatn$, $\dgcoCat$ and $\dgcoCata$). Moreover, in the graded (respectively, dg) setting, both $\coDer(\fF_1,\fF_2)$ and $\Der(\fF_1,\fF_2)$ are graded (respectively, dg) $\K$-modules in a natural way. Recall that the differential of $D\in\coDer(\fF_1,\fF_2)$ of degree $n$ is defined as $d_{\cC_2}\comp D-(-1)^nD\comp d_{\cC_1}$, and similarly for derivations. In the following we will freely use the fact that \autoref{quivadj} and \autoref{cofree} admit obvious extensions to the graded setting.

\medskip

Given $\cA\in\ACatn$, the \emph{bar construction} $\Bi(\cA)\in\dgcoCatn$ associated to $\cA$ is simply defined to be $\Tcr{\sh{\cA}}$ as a graded non unital cocategory. The differential $d_{\Bi(\cA)}$ is the degree $1$ $(\id_{\Tcr{\sh{\cA}}},\id_{\Tcr{\sh{\cA}}})$-coderivation corresponding to the maps $\m{i}{\cA}$ under the isomorphism of \autoref{cofree}. As in the case of algebras (see \cite[Section 1.2.1]{LH}) it is easy to verify that the condition $d_{\Bi(\cA)}^2=0$ corresponds precisely to the relations \eqref{eqn:catrel}.

We also set $\Bia(\cA):=\aug{\Bi(\red{\cA})}\in\dgcoCata$ for $\cA\in\ACata$. Notice that, as a graded augmented cocategory, this is just $\Tc{\sh{\red{\cA}}}$. When $\cA$ is a dg category, we will simply write $\B(\cA)$ and $\Ba(\cA)$ to simplify the notation.

If $\fF\colon\cA_1\to\cA_2$ is a non-unital $A_\infty$ functor, then we can define a dg cofunctor $\Bi(\fF)\colon\Bi(\cA_1)\to\Bi(\cA_2)$ which is the unique morphism in $\gcoCatn$ corresponding, under \autoref{quivadj}, to the morphism $\Bi(\cA_1)\to\sh{\cA_2}$ in $\gQu$ naturally defined in terms of the components $\fF^i$ of $\fF$. As in the case of algebras (see \cite[Section 1.2.1]{LH}) it is easy to verify that the condition that $\Bi(\fF)$ commutes with the differentials $d_{\Bi(\cA_i)}$ corresponds precisely to the relations \eqref{eqn:funrel}.

\begin{remark}\label{rmk:Bifun}
Note that the composition in $\ACatn$ is then defined in such a way that
\begin{equation*}
\Bi\colon\ACatn\to\dgcoCatn
\end{equation*}
defines a fully faithful functor. It is important to observe that the composition in $\ACatn$ extends the natural one in $\dgCatn$.
\end{remark}

More generally, if we consider $\cA_1,\ldots,\cA_n,\cA\in\ACatn$, we can take morphisms
\[
\fF\colon\red{\aug{\Bi(\cA_1)}\otimes\cdots\otimes\aug{\Bi(\cA_n)}}\to\sh{\cA}
\]
in $\gQu$ satisfying the natural generalization of \eqref{eqn:funrel}. More precisely, this means that the natural extension
\[
\red{\aug{\Bi(\cA_1)}\otimes\cdots\otimes\aug{\Bi(\cA_n)}}\to\Bi(\cA)
\]
of $\fF$ in $\gcoCatn$ given by \autoref{quivadj} commutes with the differentials. Such an extension will be denoted by $\Bi(\fF)$ and $\fF$ will be called an \emph{$A_\infty$ multifunctor} from $\cA_1,\ldots,\cA_n$ to $\cA$. The set of  $A_\infty$ multifunctors from $\cA_1,\ldots,\cA_n$ to $\cA$ will be denoted $\ACatn(\cA_1,\ldots,\cA_n,\cA)$. Hence $\ACatn$ has the structure of a multicategory (see \cite[Chapter 3]{BLM} for an extensive discussion about multicategories). In particular, the construction is compatible with compositions, in the sense that, for every $\fF\in\ACatn(\cA_1,\dots,\cA_n,\cA)$ and every $\fG\in\ACatn(\cA,\cA')$, one has $\Bi(\fG\comp\fF)=\Bi(\fG)\comp\Bi(\fF)$.

Given $\fF\in\ACatn(\cA_1,\ldots,\cA_n,\cA)$ and $A_i\in\cA_i$ (for $i=1,\dots,n$), one can consider the restrictions $\fF\rest{(A_1,\dots,A_{i-1},\cA_i,A_{i+1},\dots,A_n)}\colon\Bi(\cA_i)\to\sh{\cA}$, which are ordinary $A_\infty$ functors. When $\cA_1,\ldots,\cA_n,\cA$ are cohomologically unital $A_\infty$ categories, we say that $\fF$ is a \emph{cohomologically unitaly $A_\infty$ multifunctor} if all its restrictions are cohomologically unital.\footnote{Note that the definition of cohomologically unital $A_\infty$ multifunctor in \cite{BLM} is different from the one above. Nevertheless, by \cite[Proposition 9.13]{BLM}, the two definitions are equivalent.} Clearly we denote by $\ACatc(\cA_1,\ldots,\cA_n,\cA)$ the set of cohomologically unital $A_\infty$ multifunctors from $\cA_1,\ldots,\cA_n$ to $\cA$.

\subsection{Cobar construction and adjunctions}\label{subsec:cobaradj}

Given a $\cC\in\dgcoCatn$, one gets a non-unital dg category which, as a non-unital graded category, is simply defined as
\[
\coB(\cC):=\Tar{\sh[-1]{\cC}}.
\]
In other words, the objects of $\coB(\cC)$ are the same as those in $\cC$ while
\[
\coB(\cC)(C_1,C_2):=\bigoplus\cC(D_n,C_2)[-1]\otimes\cC(D_{n-1},D_n)[-1]\otimes\cdots\otimes\cC(C_1,D_1)[-1],
\]
where the sum is over all integers $n\geq 0$ and all possible $n$-uples of objects $D_1,\ldots,D_n\in\Ob{\cC}$ (if $n=0$, in the sum we only get $\cC(C_1,C_2)[-1]$). We set $\m{1}{\coB(\cC)}$ to be the derivation corresponding, under the identification given in \autoref{cofree}, to $d_\cC$ and $\comult{\cC}$. The fact that $\cC$ is a dg cocategory implies that $\m{1}{\coB(\cC)}\comp\m{1}{\coB(\cC)}=0$.

Given $\fF\colon\cC_1\to\cC_2$ in $\dgcoCatn$, one can define $\coB(\fF)\colon\coB(\cC_1)\to\coB(\cC_2)$ in $\dgCatn$ as the unique morphism in $\gCatn$ corresponding, under \autoref{quivadj}, to the morphism $\sh[-1]{\cC_1}\to\coB(\cC_2)$ in $\gQu$ naturally defined in terms of $\fF$. It turns out that the fact that $\fF$ is a dg cofunctor implies that $\coB(\fF)$ commutes with $\m{1}{\coB(\cC_i)}$. By putting altogether, we get a faithful (but non-full) functor
\[
\coB\colon\dgcoCatn\to\dgCatn.
\]

\medskip

We can now investigate the relation between the bar and cobar constructions which is based on the notion of twisting cochain, used in the proof of \autoref{barcobaradj} below. Recall that, if $\cC\in\dgcoCatn$ and $\cA\in\dgCatn$, a \emph{twisting cochain} $\tc\colon\cC\to\cA$ is a degree $1$ morphism of graded quivers such that
\[
\m{1}{\cA}\comp\tc+\tc\comp d_{\cC}+\m{2}{\cA}\comp(\tc\otimes\tc)\comp\comult{\cC}=0.
\]
When $\cC\in\dgcoCata$ and $\cA\in\dgCata$, a twisting cochain $\cC\to\cA$ is \emph{admissible} if it is of the form $\cC\epi\red{\cC}\to\red{\cA}\mono\cA$, for some map $\red{\cC}\to\red{\cA}$ (which is then necessarily a twisting cochain).

The following result is proved in \cite[Lemma 1.2.2.5]{LH} in the case of algebras and is crucial for the rest of the paper.

\begin{prop}\label{barcobaradj}
	The bar and cobar constructions define adjoint functors
	\[
	\adjpair{\coB}{\dgcoCatn}{\dgCatn}{\B}.
	\]
\end{prop}

\begin{proof}
	Given $\cC\in\dgcoCatn$ and $\cA\in\dgCatn$, in view of \autoref{quivadj} there are natural isomorphisms
	\begin{multline*}
	\gcoCatn(\cC,\B(\cA))=\gcoCatn(\cC,\Tcr{\sh{\cA}})\iso\gQu(\cC,\sh{\cA}) \\
	\iso\gQu(\sh[-1]{\cC},\cA)\iso\gCatn(\Tar{\sh[-1]{\cC}},\cA)=\gCatn(\coB(\cC),\cA).
	\end{multline*}
	It is easy to check that morphisms commuting with the differentials in the first and in the last terms correspond precisely to twisting cochains in the middle terms.
\end{proof}

We will denote by $\ncb\colon\coB\comp\B\to\id_{\dgCatn}$ the counit and by $\nbc\colon\id_{\dgcoCatn}\to\B\comp\coB$ the unit of the adjunction of \autoref{barcobaradj}. Notice that, by \autoref{rmk:Bifun} and the discussion following it, for every $\cA_1,\dots,\cA_n\in\ACatn$, the dg cofunctor
\[
\nbc[\red{\aug{\Bi(\cA_1)}\otimes\cdots\otimes\aug{\Bi(\cA_n)}}]\colon\red{\aug{\Bi(\cA_1)}\otimes\cdots\otimes\aug{\Bi(\cA_n)}}\to\B(\coB(\red{\aug{\Bi(\cA_1)}\otimes\cdots\otimes\aug{\Bi(\cA_n)}}))
\]
is of the form $\nbc[\red{\aug{\Bi(\cA_1)}\otimes\cdots\otimes\aug{\Bi(\cA_n)}}]=\Bi(\ncbi[\cA_1,\dots,\cA_n])$ for a unique $A_\infty$ multifunctor
\[
\ncbi[\cA_1,\dots,\cA_n]\in\ACatn(\cA_1,\dots,\cA_n,\coB(\red{\aug{\Bi(\cA_1)}\otimes\cdots\otimes\aug{\Bi(\cA_n)}})).
\]
Denoting by $\fdgAn\colon\dgCatn\to\ACatn$ the inclusion functor and setting
\[
\VdBn:=\coB\comp\Bi\colon\ACatn\to\dgCatn,
\]
it is clear that the $A_\infty$ functors $\ncbi[\cA]\colon\cA\to\coB(\Bi(\cA))=\VdBn(\cA)$ (for $\cA\in\ACatn$) define a natural transformation $\ncbi\colon\id_{\ACatn}\to\fdgAn\comp\VdBn$.

\begin{prop}\label{dgAnadj}
	There is an adjunction
	\[
	\adjpair{\VdBn}{\ACatn}{\dgCatn}{\fdgAn},
	\]
	whose unit is $\ncbi\colon\id_{\ACatn}\to\fdgAn\comp\VdBn$ and whose counit is $\ncb\colon\VdBn\comp\fdgAn=\coB\comp\B\to\id_{\dgCatn}$. Moreover, $\ncbi[\cA]$ (for every $\cA\in\ACatn$) and $\ncb[\cB]$ (for every $\cB\in\dgCatn$) are quasi-isomorphisms.
\end{prop}

\begin{proof}
	The first part is a very easy consequence of \autoref{barcobaradj} (again, taking into account that $\Bi$ is fully faithful). The last statement is proved in \cite[Lemma 1.3.3.6 \& Lemma 1.3.2.3]{LH}.
\end{proof}

\subsection{$A_\infty$ categories of $A_\infty$ multifunctors}\label{subsect:Ainftyfunctors}

We now want to introduce the $A_\infty$ category of $A_\infty$ multifunctors from $\cA_1,\ldots,\cA_n$ to $\cA$. For this, we start with the following.

\begin{definition}\label{def:prenat}
Given two non-unital $A_\infty$ multifunctors $\fF_1,\fF_2\in\ACatn(\cA_1,\ldots,\cA_n,\cA)$, a \emph{prenatural transformation $\theta$} of degree $d$ between $\fF_1$ and $\fF_2$ is the datum, for any pair of objects $C_1,C_2\in\red{\aug{\Bi(\cA_1)}\otimes\cdots\otimes\aug{\Bi(\cA_n)}}$, of a morphism of graded $\K$-vector spaces
\[
\theta_{C_1,C_2}\colon\aug{\Bi(\cA_1)}\otimes\cdots\otimes\aug{\Bi(\cA_n)}(C_1,C_2)\to\cA(\fF_1(C_1),\fF_2(C_2))
\]
of degree $d$.
\end{definition}

Given $\cA_1,\ldots,\cA_n,\cA\in\ACatn$, it is explained in \cite[Proposition 8.15]{BLM} that one has a non-unital $A_\infty$ category denoted by $\FunAn(\cA_1,\ldots,\cA_n,\cA)$ whose objects are non-unital $A_\infty$ multifunctors between $\cA_1,\ldots,\cA_n$ and $\cA$ and morphisms are prenatural trasformations of all degrees. If $\cA_1,\ldots,\cA_n,\cA\in\ACatc$, we denote by $\FunAc(\cA_1,\ldots,\cA_n,\cA)$ the full $A_\infty$ subcategory of $\FunAn(\cA_1,\ldots,\cA_n,\cA)$ whose objects are cohomologically unital $A_\infty$ multifunctors.

\begin{remark}\label{rmk:unital}
If $\cA_1$ and $\cA_2$ are strictly unital $A_\infty$ categories one has the full $A_\infty$ subcategory $\FunA(\cA_1,\cA_2)$ of $\FunAn(\cA_1,\cA_2)$ whose objects are strictly unital $A_\infty$ functors.

On the other hand, if $\fF_1,\fF_2\colon\cA_1\to\cA_2$ are strictly unital $A_\infty$ functors, one can consider
\emph{strictly unital prenatural transformations} $\theta$ between them. Recall that $\theta$ is strictly unital if $\theta(f_i,\ldots,\id_A,\ldots,f_1)=0$, for all $i\geq 1$, all $A\in\cA_1$ and all morphisms $f_1,\ldots,f_i$. One could then define $\FunA^u(\cA_1,\cA_2)$ as the $A_\infty$ subcategory of $\FunA(\cA_1,\cA_2)$ with the same objects and morphisms the strictly unital prenatural transformations. By \cite[Lemma 8.2.1.3]{LH}, the natural inclusion $\FunA^u(\cA_1,\cA_2)\mono\FunA(\cA_1,\cA_2)$ is a quasi-isomorphism. Thus we will only work with the latter.
\end{remark}

\begin{remark}\label{natcoder}
Given $\fF_1,\fF_2\in\ACatn(\cA_1,\ldots,\cA_n,\cA)$, the graded $\K$-vector space $\Nat(\fF_1,\fF_2)$ consisting of prenatural transformations between $\fF_1$ and $\fF_2$ is isomorphic to $\coDer(\Bi(\fF_1),\Bi(\fF_2))$ by \autoref{cofree}. It follows from the discussion in \cite[Section 8.16]{BLM} (see also the beginning of \cite[Section 3.3]{Ma} for the case $n=1$) that this is actually an isomorphism of dg $\K$-vector spaces if we endow $\Nat(\fF_1,\fF_2)$ with the differential given by $\m{1}{\FunAn(\cA_1,\ldots,\cA_n,\cA)}$.  
\end{remark}

Given $\fF\in\ACatn(\cA_1,\dots,\cA_n,\cA)$ and $\cA'\in\ACatn$, there is an induced non-unital $A_\infty$ functor
\[
\rcf{\fF}\colon\FunAn(\cA,\cA')\to\FunAn(\cA_1,\dots,\cA_n,\cA')
\]
defined on objects by $\fG\mapsto\fG\comp\fF$. Moreover, if $\cA_1,\ldots,\cA_n,\cA,\cA'$ and $\fF$ are cohomologically unital, then $\rcf{\fF}$ restricts to a non-unital $A_\infty$ functor
\[
\rcf{\fF}\colon\FunAc(\cA,\cA')\to\FunAc(\cA_1,\dots,\cA_n,\cA').
\]

It is observed for example in \cite{BLM} that, if $\cA$ is a cohomologically unital (resp.\ strictly unital) $A_\infty$ category, then $\FunAn(\cA_1,\ldots,\cA_n,\cA)$ is a cohomologically unital (resp.\ strictly unital) $A_\infty$ category. Similarly, if $\cA$ is a dg category, then $\FunAn(\cA_1,\ldots,\cA_n,\cA)$ is a dg category as well. Moreover, if $\cA_1,\ldots,\cA_n$ are in $\ACatc$ and $\cA$ is in $\ACatc$ (resp. in $\ACat$ or in $\dgCat$), then $\FunAc(\cA_1,\ldots,\cA_n,\cA)$  is in $\ACatc$ (resp. in $\ACat$ or in $\dgCat$). Finally, if $\cA_1,\ldots,\cA_n$ are in $\ACat$  and $\cA$ is in $\ACat$ (resp. in $\dgCat$), then $\FunA(\cA_1,\ldots,\cA_n,\cA)$ is in $\ACat$ (resp. in $\dgCat$).

We also need to recall the following notions.

\begin{definition}
Let $\fF_1,\fF_2\in\ACatn(\cA_1,\ldots,\cA_n,\cA)$.	

(i) When $\cA$ is cohomologically unital, $\fF_1$ and $\fF_2$ are \emph{weakly equivalent} (denoted by $\fF_1\hison\fF_2$) if they are isomorphic in $H^0(\FunAn(\cA_1,\ldots,\cA_n,\cA))$.

(ii) When $n=1$ and $\fF_1^0=\fF_2^0$, $\fF_1$ and $\fF_2$ are \emph{homotopic} (denoted by $\fF_1\htpn\fF_2$) if there is a prenatural transformation $\theta$ of degree $0$ such that $\theta(\id_A)=0$, for every $A\in\cA_1$, and $\fF_2=\fF_1+\m{1}{\FunAn(\cA_1,\cA_2)}(\theta)$.
\end{definition}

Note that, if $\fF_1\htpn\fF_2$ and $\cA$ is cohomologically unital, then $\fF_1\hison\fF_2$ by \cite[Lemma 2.5]{Sei}.

Since $\hison$ is compatible with compositions, from the (multi)category $\ACatc$ one can obtain a quotient (multi)category $\QACatc$ with the same objects and whose morphisms are given by
\[
\QACatc(\cA_1,\ldots,\cA_n,\cA):=\ACatc(\cA_1,\ldots,\cA_n,\cA)/\hison.
\]
Similarly one can construct $\QACat$ from $\ACat$.

Although we will not need it in the rest of the paper, it is interesting to point out the following consequence of \autoref{dgAnadj}.

\begin{cor}\label{Ahtpdg}
For every $\cC\in\dgcoCatn$, every $\cA\in\dgCatn$ and every $\fF\in\ACatn(\coB(\cC),\cA)$, there exists $\fF'\in\dgCatn(\coB(\cC),\cA)$ such that $\fF\htpn\fF'$.
\end{cor}

\begin{proof}
We define $\fF'$ as the composition
\[
\fF'\colon\coB(\cC)\mor{\coB(\nbc[\cC])}\coB(\B(\coB(\cC)))=\VdBn(\coB(\cC))\mor{\VdBn(\fF)}\VdBn(\cA)\mor{\ncb[\cA]}\cA.
\]
Notice that $\fF'$ is a (non-unital) dg functor by construction. Since $\ncb[\cA]\comp\ncbi[\cA]=\id_{\cA}$ (this is due to the fact that the composition
\[
\fdgAn(\cA)\mor{\ncbi[\fdgAn(\cA)]}\fdgAn(\VdBn(\fdgAn(\cA)))\mor{\fdgAn(\ncb[\cA])}\fdgAn(\cA)
\]
is the identity), we have
\[
\fF=\ncb[\cA]\comp\ncbi[\cA]\comp\fF=\ncb[\cA]\comp\VdBn(\fF)\comp\ncbi[\coB(\cC)].
\]
Thus, to conclude that $\fF\htpn\fF'=\ncb[\cA]\comp\VdBn(\fF)\comp\coB(\nbc[\cC])$, it is enough to prove that $\ncbi[\coB(\cC)]\htpn\coB(\nbc[\cC])$. To see this, just observe that $\ncb[\coB(\cC)]\colon\coB(\B(\coB(\cC)))=\VdBn(\coB(\cC))\to\coB(\cC)$ satisfies both $\ncb[\coB(\cC)]\comp\ncbi[\coB(\cC)]=\id_{\coB(\cC)}$ (similarly as above) and $\ncb[\coB(\cC)]\comp\coB(\nbc[\cC])=\id_{\coB(\cC)}$ (thanks to the adjunction between $\coB$ and $\B$). Since $\ncb[\coB(\cC)]$ is a quasi-isomorphism and quasi-isomorphisms are invertible up to homotopy (by \cite[Corollary 1.3.1.3]{LH}), this clearly implies that $\ncbi[\coB(\cC)]\htpn\coB(\nbc[\cC])$.
\end{proof}

\subsection{No model structure}\label{subsec:nomodel}

We conclude this section with a side remark that shows that there is no model structure for the category of $A_\infty$ categories.

It was proven in \cite{Ta} that $\dgCat$ has a (complete and cocomplete) model structure (the reader can have a look at \cite{Hov} for a thorough exposition about model categories). On the other hand, it was stated in \cite{LH} that the category of $A_\infty$ algebras does not carry a model structure, according to the modern terminology, in the sense that the category of $A_\infty$ algebras is not closed under limits and colimits.

As no explicit example supporting the latter statement is provided in \cite{LH} or, to our knowledge, in the existing literature, we sketch here a simple example in the more general setting of $A_\infty$ categories that may be very well known to experts.

Consider the following (unital) graded $\K$-algebras:
\[
\cA:=\K[\varepsilon_1,\varepsilon_2]\qquad\cA':=\K[\varepsilon],
\]
where $\varepsilon_1^2=\varepsilon_2^2=\varepsilon_1\varepsilon_2=\varepsilon_2\varepsilon_1=\varepsilon^2=0$, $\deg(\varepsilon_1)=\deg(\varepsilon_2)=0$ and $\deg(\varepsilon)=-1$. Being dg algebras (with zero differential), we can regard $\cA$ and $\cA'$ as strictly unital $A_\infty$ categories. Similarly, we will consider the dg morphism $\fF_1\colon\cA\to\cA'$ defined by
\[
\fF_1(\varepsilon_1)=\fF_1(\varepsilon_2)=0
\]
as a strict and strictly unital $A_\infty$ morphism. On the other hand, it is easy to check that one can define another strictly unital $A_\infty$ morphism $\fF_2\colon\cA\to\cA'$ by setting $\fF_2^1=\fF^1_1$,
\[
\fF_2^2(\varepsilon_i,\varepsilon_j)=
\begin{cases}
\varepsilon & \text{if $i=1$, $j=2$} \\
0 & \text{otherwise}
\end{cases}
\]
and $\fF_2^i=0$ for $i>2$.

\begin{lem}\label{lem:equal}
Let $\cA$, $\cA'$, $\fF_1$ and $\fF_2$ be as above. Then $\fF_1$ and $\fF_2$ do not admit an equalizer in $\ACat$. Moreover, they do not admit an equalizer in $\ACatc$, either.
\end{lem}

\begin{proof}
Assume, for a contradiction, that $\fG\colon\cB\to\cA$ is an equalizer of $\fF_1$ and $\fF_2$ in $\ACat$ (respectively, in $\ACatc$).

Let us first observe that for every $B\in\Ob{\cB}$, there is a strictly (respectively, cohomologically) unital $A_\infty$ functor $\fF_B\colon\K\to\cB$ sending the unique object of $\K$ to $B$. Indeed, this is trivial if $\cB$ is strictly unital. Otherwise, we can use the fact that, by \cite[Lemma 2.1]{Sei}, there is a quasi-isomorphism $\fI\in\ACatc(\cB',\cB)$ with $\fI^0$ the identity, for some $\cB'\in\ACat$. Then we can take as $\fF_B$ the composition $\fI\comp\fF'_B$, where $\fF'_B\colon\K\to\cB'$ is a strictly unital $A_\infty$ functor sending the unique object of $\K$ to $B$.

Next we prove that also $\cB$ has only one object. To see this, given $B_1,B_2\in\Ob{\cB}$, let $\fF_{B_1}$ and $\fF_{B_2}$ be the $A_\infty$ functors introduced above. Since $\K$ and $\cA$ sit in degree $0$, the $A_\infty$ functor $\fG\comp\fF_{B_i}\colon\K\to\cA$ must be strict and $H(\fG\comp\fF_{B_i})$ can be identified with $(\fG\comp\fF_{B_i})^1$, for $i=1,2$. As $\fG\comp\fF_{B_i}$ is cohomologically unital, this implies that in any case it is also strictly unital. Since the source category is $\K$ and the target $\cA$ has only one object, we immediately deduce that $\fG\comp\fF_{B_1}=\fG\comp\fF_{B_2}$. As $\fG$ is a monomorphism in $\ACat$ (respectively, in $\ACatc$), this implies $\fF_{B_1}=\fF_{B_2}$, hence $B_1=B_2$, as wanted.

Thus $\cB$ is an $A_\infty$ algebra and $\fG$ is a morphism of $A_\infty$ algebras, which is an equalizer of $\fF_1$ and $\fF_2$ also in the category of strictly (respectively, cohomologically) unital $A_\infty$ algebras. We are going to see that this last statement gives a contradiction.

To this purpose, consider the (unital) $\K$-algebra $\cA_0:=\K[\varepsilon_0]$, where $\varepsilon_0^2=0$ and $\deg(\varepsilon_0)=0$. For $i=1,2$, we denote by $\fJ_i\colon\cA_0\to\cA$ the morphism of $\K$-algebras (which we regard, as usual, as a strict and strictly unital $A_\infty$ morphism) defined by $\fJ_i(\varepsilon_0)=\varepsilon_i$. It is clear that $\fF_1\comp\fJ_i=\fF_2\comp\fJ_i$. Hence, by the universal property of the equalizer $\fG$, we have that $\fJ_i$ factors through $\fG$. Since the images of $\fJ_1$ and $\fJ_2$ together span $\cA$ as a $\K$-vector space, we obtain that $\fG^1$ is surjective. From this it easy to see that $(\fF_1\comp\fG)^2\ne(\fF_2\comp\fG)^2$, which is impossible because $\fF_1\comp\fG=\fF_2\comp\fG$.
\end{proof}

The following is then straightforward.

\begin{prop}\label{prop:nomodel}
The categories $\ACat$ and $\ACatc$ do not admit a model structure.
\end{prop}

The approach in \cite{LH} consists then in taking a reduced model structure on the category of $A_\infty$ algebras which does not require the existence of arbitrary limits and colimits. It would be interesting but maybe more difficult to exploit the same idea for $A_\infty$ categories. As this is not needed for the results in this paper, we leave it for future work.

\section{Equivalences of localizations}\label{sect:proofthm1}

This section deals with the proof of \autoref{thm:localizations} concerning the localizations of the categories of $A_\infty$ categories and their comparison with the localization of the category of dg categories.

\subsection{Equivalence between $\Hqe$ and $\HoACat$}

We will need the following unital version of \autoref{dgAnadj}.

\begin{prop}\label{dgAadj}
There is an adjunction
\[
\adjpair{\VdB}{\ACat}{\dgCat}{\fdgA},
\]
where $\fdgA$ is the inclusion functor. Moreover, the unit $\ntV\colon\id_{\ACat}\to\fdgA\comp\VdB$ and the counit $\nfV\colon\VdB\comp\fdgA\to\id_{\dgCat}$ are such that $\ntV[\cA]$ (for every $\cA\in\ACat$) and $\nfV[\cB]$ (for every $\cB\in\dgCat$) are quasi-isomorphisms.
\end{prop}

\begin{proof}
Keeping the notation of \autoref{subsec:cobaradj}, for every $\cA\in\ACat$ let $\Vid{\cA}$ be the smallest dg ideal in $\aug{\VdBn(\cA)}$ such that the $A_{\infty}$ functor
\[
\cA\mor{\ncbi[\cA]}\VdBn(\cA)\mono\aug{\VdBn(\cA)}\epi\aug{\VdBn(\cA)}/\Vid{\cA}
\]
is strictly unital. We then define $\VdB(\cA):=\aug{\VdBn(\cA)}/\Vid{\cA}$ and denote the above composition by $\ntV[\cA]\colon\cA\to\VdB(\cA)=\fdgA(\VdB(\cA))$. We claim that $\ntV[\cA]$ is a universal arrow from $\cA$ to $\fdgA$, namely that, for every $\cB\in\dgCat$, the map
\begin{equation}\label{e:dgAadj}
\begin{split}
\dgCat(\VdB(\cA),\cB) & \to \ACat(\cA,\cB) \\
\fG & \mapsto \fG\comp\ntV[\cA]
\end{split}
\end{equation}
is bijective. Indeed, given $\fF\in\ACat(\cA,\cB)$, we can regard $\fF$ as an element of $\ACatn(\cA,\cB)$. By \autoref{dgAnadj}, there exists unique $\fF'\in\dgCatn(\VdBn(\cA),\cB)$ such that $\fF=\fF'\comp\ncbi[\cA]$, and $\fF'$ extends uniquely to $\fF''\in\dgCat(\aug{\VdBn(\cA)},\cB)$. It follows easily from the definition of $\Vid{\cA}$ that $\Vid{\cA}\subseteq\ker(\fF'')$; hence $\fF''$ factors uniquely through a dg functor $\fG\colon\aug{\VdBn(\cA)}/\Vid{\cA}\to\cB$. It is then clear that $\fG\in\dgCat(\VdB(\cA),\cB)$ is unique such that $\fF=\fG\comp\ntV[\cA]$. Thus, by \cite[Theorem 2, p.\ 83]{M}, the map on objects $\cA\mapsto\VdB(\cA)$ extends uniquely to a functor $\VdB\colon\ACat\to\dgCat$, which is left adjoint to $\fdgA$ and such that $\ntV\colon\id_{\ACat}\to\fdgA\comp\VdB$ is the unit of the adjunction.

To conclude, it is enough to show that $\ntV[\cA]$ is a quasi-isomorphism for every $\cA\in\ACat$. Indeed, assuming this, $\ntV[\cB]=\ntV[\fdgA(\cB)]$ is a quasi-isomorphism for every $\cB\in\dgCat$. Since the composition
\[
\fdgA(\cB)\mor{\ntV[\fdgA(\cB)]}\fdgA(\VdB(\fdgA(\cB)))\mor{\fdgA(\nfV[\cB])}\fdgA(\cB)
\]
is the identity, it follows that $\fdgA(\nfV[\cB])$ (whence $\nfV[\cB]$) is a quasi-isomorphism, too.

The argument to show that $\ntV[\cA]$ is a quasi-isomorphism seems to be well known to experts and it was pointed out to us by Michel Van den Bergh. We outline it here for the convenience of the reader.

Given $A_1,A_2\in\cA$, we first give a filtration on the graded vector space $\cA(A_1,A_2)$:
\[
F^n\cA(A_1,A_2)=\begin{cases}\K\id_{A_1}&\text{if $n=0$ and $A_1=A_2$}\\ 0&\text{if $n=0$ and $A_1\neq A_2$}\\\cA(A_1,A_2)&\text{otherwise.}\end{cases}
\]
Similarly, we define a filtration on $\VdB(\cA)(A_1,A_2)$ given by
\[
F^0\VdB(\cA)(A_1,A_2)=\begin{cases}\K\id_{A_1}&\text{if $A_1=A_2$}\\ 0&\text{if $A_1\neq A_2$}\end{cases}
\]
while, for $n>0$, we set $F^n\VdB(\cA)(A_1,A_2)$ to be the vector space freely generated by the elements of the form
\[
(f_1|\cdots|f_{i_1})\otimes(f_{i_1+1}|\cdots|f_{i_2})\otimes\cdots\otimes(f_{i_{m-1}+1}|\cdots|f_{k}),
\]
for all $k\leq n$ and all morphisms $f_1,\ldots,f_{i_1},f_{i_1+1},\ldots,f_{i_2},\ldots,f_{i_{m-1}+1},\ldots,f_{k}$ in $\cA$. Here we denote by $|$ the tensor product in the cocategory while $\otimes$ stays for the tensor product in the category.

It is easy to check that these filtrations are compatible with the differential and with $\ntV[\cA]$. This clearly implies that $\ntV[\cA]$ is a quasi-isomorphism if the induced morphism (of dg quivers) $\gr{\ntV[\cA]}\colon\gr{\cA}\to\gr{\VdB(\cA)}$ is a quasi-isomorphism. Since the differential on $\gr{\VdB(\cA)}$ (which is induced by the one on $\VdB(\cA)$) depends only on $\m{1}{\cA}$, we can assume that $\cA$ is augmented with $\m{i}{\red{\cA}}=0$ for  $i>1$. Under this assumption $\gr{\VdB(\cA)}$ can be identified with $\aug{\VdBn(\red{\cA})}$ and $\gr{\ntV[\cA]}$ with $\aug{\ncbi[\red{\cA}]}$. Since $\ncbi[\red{\cA}]\colon\red{\cA}\to\VdBn(\red{\cA})$ is a quasi-isomorphism by \autoref{dgAnadj}, we conclude that $\gr{\ntV[\cA]}$ is a quasi-isomorphism, as well.
\end{proof}

From \autoref{dgAadj} it is easy to deduce the following result.

\begin{thm}\label{thm:1}
The functors $\fdgA$ and $\VdB$ induce the functors
\[
\Ho{\fdgA}\colon\Hqe\to\HoACat\qquad \text{and}\qquad\Ho{\VdB}\colon\HoACat\to\Hqe
\]
which are quasi-inverse equivalences of categories.
\end{thm}

\begin{proof}
Obviously $\fdgA\colon\dgCat\to\ACat$ preserves quasi-equivalences, and the same is true for $\VdB\colon\ACat\to\dgCat$. Indeed, for every $\fF\colon\cA\to\cA'$ in $\ACat$ we have $\VdB(\fF)\comp\ntV[\cA]=\ntV[\cA']\comp\fF$. Since $\ntV[\cA]$ and $\ntV[\cA']$ are quasi-isomorphisms, this immediately implies that $\VdB(\fF)$ is a quasi-equivalence if $\fF$ is such. Thus, by the universal property of localization, $\fdgA$ and $\VdB$ induce the required functors $\Ho{\fdgA}$ and $\Ho{\VdB}$. Moreover, $\ntV$ and $\nfV$ induce natural transformations
\begin{gather*}
\nhtV\colon\Ho{\id_{\ACat}}=\id_{\Ho{\ACat}}\to\Ho{\fdgA\comp\VdB}=\Ho{\fdgA}\comp\Ho{\VdB}, \\
\nhfV\colon\Ho{\VdB\comp\fdgA}=\Ho{\VdB}\comp\Ho{\fdgA}\to\Ho{\id_{\dgCat}}=\id_{\Ho{\dgCat}}
\end{gather*}
such that $\nhtV[\cA]$ (for every $\cA\in\ACat$) and $\nhfV[\cB]$ (for every $\cB\in\dgCat$) are just the images in the localizations of $\ntV[\cA]$ and $\nfV[\cB]$. As $\ntV[\cA]$ and $\nfV[\cB]$ are quasi-isomorphisms, this proves that $\nhtV$ and $\nhfV$ are isomorphisms of functors.
\end{proof}

\begin{remark}\label{ACatdg}
Defining $\ACatdg$ to be the full subcategory of $\ACat$ with objects the dg categories, the above argument can be adapted in an obvious way to prove that there is an equivalence of categories between $\Hqe$ and $\HoACatdg$ (the latter category denoting, of course, the localization of $\ACatdg$ with respect to quasi-equivalences). Hence $\HoACat$ and $\HoACatdg$ are equivalent, as well.
\end{remark}

\subsection{Equivalence between $\HoACat$ and $\HoACatc$}

We start discussing some preliminary results that show how to replace a cohomologically unital $A_\infty$ category or functor with a strictly unital one.

\begin{prop}\label{catcu}
Given $\cA\in\ACatc$, there exists $\cA'\in\ACat$ such that $\cA\iso\cA'$ in $\ACatc$.
\end{prop}

\begin{proof}
See \cite[Lemma 2.1]{Sei}.
\end{proof}

The proof of the above result, which is crucial for the rest of the paper, secretly uses a reduction to \emph{minimal models} for $A_\infty$ categories. Such models do not exist in general for $A_\infty$ categories defined over commutative rings which are not fields.

We now want to state a result similar to the above one but for cohomologically unital $A_\infty$ functors, which is probably well known to experts. To this purpose, we will denote by $\ACatcu$ the full subcategory of $\ACatc$ with the same objects as $\ACat$, and, as usual, by $\HoACatcu$ its localization with respect to quasi-equivalences.

\begin{prop}\label{funcu}
Given $\fF\in\ACatcu(\cA_1,\cA_2)$, there exists $\fF'\in\ACat(\cA_1,\cA_2)$ such that $\fF\htpn\fF'$.
\end{prop}

\begin{proof}
This is mentioned in \cite[Remark 2.2]{Sei} without a proof, since one could give an explicit one which is very similar, in spirit, to that of \cite[Lemma 2.1]{Sei}. Alternatively, the result follows from \cite[Theorem 3.2.2.1]{LH} in the case of minimal $A_\infty$ algebras. The passage to arbitrary $A_\infty$ algebras follows from \cite[Proposition 3.2.4.3]{LH}. This also covers the case of $A_\infty$ categories as explained in \cite[Section 5.1]{LH}. For all the details of the proof, see also \cite[Section 3.2]{Orn}.
\end{proof}

The following consequence will be used later.

\begin{cor}\label{lem:qefun}
For every $\cA_1,\cA_2\in\ACat$ there is an isomorphism in $\HoACat$
\[
\FunA(\cA_1,\cA_2)\iso\FunAc(\cA_1,\cA_2).
\]
\end{cor}

\begin{proof}
The fully faithful inclusion functor $\FunA(\cA_1,\cA_2)\mono\FunAc(\cA_1,\cA_2)$ is actually a quasi-equivalence by \autoref{funcu}.
\end{proof}

We also need the following preliminary result.

\begin{lem}\label{htpHo}
If $\fF,\fF'\in\ACatcu(\cA_1,\cA_2)$ are such that $\fF\htpn\fF'$, then $\fF$ and $\fF'$ have the same image in $\HoACatcu$. If moreover $\fF$ and $\fF'$ are strictly unital, then $\fF$ and $\fF'$ have the same image also in $\HoACat$.
\end{lem}

\begin{proof}
One can use the argument of \cite[Remark 1.11]{Sei}. Namely, there exists an $A_\infty$ functor $\fG\colon\cA_1\to I\otimes\cA_2$ such that $\fF=(f_0\otimes\id_{\cA_2})\comp\fG$ and $\fF'=(f_1\otimes\id_{\cA_2})\comp\fG$. Here $I$ is a suitably defined dg algebra and $f_i\colon I\to\K$ (for $i=0,1$) is a quasi-isomorphism in $\dgAlg$ such that $f_i\comp\strmor{I}=\id_{\K}$ (where $\strmor{I}\colon\K\to I$ is the structure morphism). Moreover, everything is in $\ACatcu$, and even in $\ACat$ if $\fF$ and $\fF'$ are strictly unital (as in this case we can assume that also the homotopy is strictly unital, by \cite[Proposition 3.2.4.3]{LH}). Writing $\loccl{\fH}$ for the image of $\fH$ in the localization, in both cases it follows that
\[
\loccl{\fF}=\loccl{f_0\otimes\id_{\cA_2}}\comp\loccl{\fG}=\loccl{\strmor{I}\otimes\id_{\cA_2}}^{-1}\comp\loccl{\fG}=\loccl{f_1\otimes\id_{\cA_2}}\comp\loccl{\fG}=\loccl{\fF'}.
\]
This concludes the proof.
\end{proof}

We can now prove another part of \autoref{thm:localizations}.

\begin{thm}\label{thm:2}
The inclusion functor $\ACat\to\ACatc$ induces an equivalence of categories $\HoACat\to\HoACatc$.
\end{thm}

\begin{proof}
By \autoref{catcu} the inclusion functor $\ACatcu\to\ACatc$ is an equivalence, and it follows easily that the same is true for the induced functor $\HoACatcu\to\HoACatc$. Therefore, it is enough to prove that the inclusion functor $\fucu\colon\ACat\to\ACatcu$ induces an equivalence (actually even an isomorphism of categories) $\Ho{\fucu}\colon\HoACat\to\HoACatcu$. To this purpose, we can define a functor $\fcuu\colon\ACatcu\to\HoACat$ as follows. Of course, $\fcuu$ is the identity on objects, whereas for every morphism $\fF$ in $\ACatcu$ we define $\fcuu(\fF)$ to be the image in $\HoACat$ of a morphism $\fF'$ in $\ACat$ whose existence is ensured by \autoref{funcu}. Notice that $\fcuu(\fF)$ is well defined by (the second part of) \autoref{htpHo}. It is straightforward to check that $\fcuu$ is really a functor and that it send quasi-equivalences to isomorphisms, thus inducing a functor $\fhcuu\colon\HoACatcu\to\HoACat$. It is clear that $\fhcuu\comp\Ho{\fucu}=\id_{\HoACat}$, and it follows from (the first part of) \autoref{htpHo} that $\Ho{\fucu}\comp\fhcuu=\id_{\HoACatcu}$.
\end{proof}

\subsection{Equivalence between $\HoACat$ and $\QACat$}

We conclude this section with the last comparisons in \autoref{thm:localizations}.

\begin{prop}\label{ucueq}
The inclusion functor $\ACat\to\ACatc$ induces an equivalence of categories $\QACat\to\QACatc$.
\end{prop}

\begin{proof}
Clearly the inclusion functor $\ACat\to\ACatc$ induces a faithful functor $\QACat\to\QACatc$. This functor is full by \autoref{funcu} and it is essentially surjective by \autoref{catcu}.
\end{proof}

\begin{lem}\label{Genovese}
If $\fF,\fF'\in\ACatdg(\cA_1,\cA_2)$ are such that $\fF\hison\fF'$, then $\fF$ and $\fF'$ have the same image in $\HoACatdg$.
\end{lem}

\begin{proof}
By \cite[Lemma 4.8]{Ge} the existence of a natural transformation $\fF\to\fF'$ (which we can assume to be strictly unital, as it is explained in \autoref{rmk:unital}) implies the existence of $\fG\in\ACatdg(\cA_1,\dgMor{\cA_2})$ such that $\fF=\fS\comp\fG$ and $\fF'=\fT\comp\fG$. Moreover, the fact that $\fF\hison\fF'$ actually implies that the image of $\fG$ is contained in the full dg subcategory $\pob{\cA_2}$ of $\dgMor{\cA_2}$ whose objects are homotopy equivalences (see also \cite[Section 2.2]{CS}). To conclude it is enough to note that $\fS,\fT\colon\pob{\cA_2}\to\cA_2$ are quasi-equivalences with the same images in $\HoACatdg$ (they both coincide with the inverse of the image of the natural dg functor $\cA_2\to\pob{\cA_2}$).
\end{proof}

The following result concludes the proof of \autoref{thm:localizations}.

\begin{thm}\label{thm:3}
The categories $\HoACat$ and $\QACat$ are equivalent.
\end{thm}

\begin{proof}
By \autoref{ACatdg} $\HoACat$ is equivalent to $\HoACatdg$. Similarly, $\QACat$ is equivalent to $\QACatdg$. To see this, just observe that for every $\cA\in\ACat$ there exists a quasi-equivalence $\cA\to\cA'$ with $\cA'\in\ACatdg$ (thanks to \autoref{dgAadj}) and that the image of a quasi-equivalence of $\ACat$ is an isomorphism in $\QACat$ (by \cite[Th\'eor\`eme 9.2.0.4]{LH}). Thus it is enough to show that $\HoACatdg$ and $\QACatdg$ are equivalent. Now, we have just noted that the quotient functor $\ACatdg\to\QACatdg$ sends quasi-equivalences to isomorphisms, hence it induces a functor $\HoACatdg\to\QACatdg$. On the other hand, it follows from \autoref{Genovese} that the localization functor $\ACatdg\to\HoACatdg$ factors through a functor $\QACatdg\to\HoACatdg$. It is then clear that in this way we obtain two functors which are quasi-inverse equivalences of categories between $\HoACatdg$ and $\QACatdg$.
\end{proof}

\section{Internal Homs via $A_\infty$ functors}\label{sect:internalHoms}

This section is devoted to the proof of \autoref{thm:internalHoms}. Our approach requires the use of multifunctors and a slightly delicate analysis of the behaviour of the bar construction under tensor product.

\subsection{Multifunctors and a useful quasi-equivalence}\label{subsect:bifun}

We now prove some preliminary results which play an important role in our proof.

\begin{lem}\label{multifundg}
Given $\fM\in\ACatn(\cA_1,\dots,\cA_n,\cA)$ with $\cA_1,\dots,\cA_n\in\ACatn$ and $\cA\in\dgCatn$, there exists unique $\dgmf{\fM}\in\dgCatn(\coB(\red{\aug{\Bi(\cA_1)}\otimes\cdots\otimes\aug{\Bi(\cA_n)}}),\cA)$ such that $\fM=\dgmf{\fM}\comp\ncbi[\cA_1,\dots,\cA_n]$. Moreover, $\coB(\Bi(\fM))$ is a quasi-isomorphism if and only if $\dgmf{\fM}$ is a quasi-isomorphism.
\end{lem}

\begin{proof}
Setting $\cC:=\red{\aug{\Bi(\cA_1)}\otimes\cdots\otimes\aug{\Bi(\cA_n)}}$, by definition $\Bi(\fM)\in\dgcoCatn(\cC,\B(\cA))$. By \autoref{barcobaradj} there exists unique $\dgmf{\fM}\in\dgCatn(\coB(\cC),\cA)$ such that $\Bi(\fM)=\B(\dgmf{\fM})\comp\nbc[\cC]$. As $\Bi$ is bijective and
\[
\B(\dgmf{\fM})\comp\nbc[\cC]=\Bi(\dgmf{\fM})\comp\Bi(\ncbi[\cA_1,\dots,\cA_n])=\Bi(\dgmf{\fM}\comp\ncbi[\cA_1,\dots,\cA_n]),
\]
this proves the first statement. As for the last one, observe that in any case $\coB(\Bi(\ncbi[\cA_1,\dots,\cA_n]))=\coB(\nbc[\cC])$ is a quasi-isomorphism (since $\ncb[\coB(\cC)]\comp\coB(\nbc[\cC])=\id_{\coB(\cC)}$ and $\ncb[\coB(\cC)]$ is a quasi-isomorphism, by \autoref{dgAnadj}). Therefore, $\coB(\Bi(\fM))=\coB(\Bi(\dgmf{\fM}))\comp\coB(\Bi(\ncbi[\cA_1,\dots,\cA_n]))$ is a quasi-isomorphism if and only if $\coB(\Bi(\dgmf{\fM}))=\VdBn(\dgmf{\fM})$ is a quasi-isomorphism. Finally, \autoref{dgAnadj} easily implies that $\VdBn(\dgmf{\fM})$ is a quasi-isomorphism if and only if $\dgmf{\fM}$ is a quasi-isomorphism.
\end{proof}

\begin{remark}\label{dercoder}
	For $\cA_1,\dots,\cA_n\in\ACatn$,  $\cA\in\dgCatn$ and $\fF_1,\fF_2\in\ACatn(\cA_1,\dots,\cA_n,\cA)$, it is easy to see that there is a natural isomorphism of dg $\K$-modules 
	\[
	\coDer(\Bi(\fF_1),\Bi(\fF_2))\iso\Der(\dgmf{\fF_1},\dgmf{\fF_2}).
	\]
\end{remark}

\begin{lem}\label{weakequiv}
Given $\cA_1,\cA_2\in\dgCatn$, there exists
\[
\fG\colon\red{\aug{\B(\cA_1)}\otimes\aug{\B(\cA_2)}}\to\B(\red{\aug{\cA_1}\otimes\aug{\cA_2}})
\]
in $\dgcoCatn$ which is the identity on objects and such that $\coB(\fG)$ is a quasi-isomorphism.
\end{lem}

\begin{proof}
Equivalently, we have to prove that there exists $\aug{\fG}\colon\Ba(\aug{\cA_1})\otimes\Ba(\aug{\cA_2})\to\Ba(\aug{\cA_1}\otimes\aug{\cA_2})$ in $\dgcoCata$ which is the identity on objects and such that $\coBa(\aug{\fG})$ is a quasi-isomorphism.

We start by recalling that, as it is proved in \cite[Section 2.2.1]{LH}, for every $\cA\in\dgCata$, the natural map
\[
\utc{\cA}\colon\Ba(\cA)=\Tc{\sh{\red{\cA}}}\epi\sh{\red{\cA}}\isomor\red{\cA}\mono\cA
\]
(where $\sh{\red{\cA}}\isomor\red{\cA}$ is the obvious degree $1$ isomorphism) is an admissible twisting cochain (see \autoref{subsec:barcobar} for the definition of admissible twisting cochain). Moreover, $\utc{\cA}$ is universal among admissible twisting cochains with target $\cA$, meaning that for every admissible twisting cochain $\tc\colon\cC\to\cA$ (for some $\cC\in\dgcoCata$) there exists a unique morphism $\fG_{\tc}\colon\cC\to\Ba(\cA)$ in $\dgcoCata$ such that $\tc=\utc{\cA}\comp\fG_{\tc}$. 

By the argument in \cite[Section 2.5.2]{LH}, there exists an admissible twisting cochain 
\[
\tc'\colon\Ba(\aug{\cA_1})\otimes\Ba(\aug{\cA_2})\to\cA':=\aug{\VdBn(\cA_1)}\otimes\aug{\VdBn(\cA_2)}
\]
which is \emph{acyclic}. By \cite[Proposition 2.2.4.1]{LH} this means that $\coBa(\fG_{\tc'})$ is a quasi-isomorphism in $\dgCata$. Setting also $\cA:=\aug{\cA_1}\otimes\aug{\cA_2}$ and $\fF:=\aug{\ncb[\cA_1]}\otimes\aug{\ncb[\cA_1]}\colon\cA'\to\cA$ in $\dgCata$, the composition $\tc:=\fF\comp\tc'\colon\Ba(\aug{\cA_1})\otimes\Ba(\aug{\cA_2})\to\cA$ is again an admissible twisting cochain. Taking into account that $\fF\comp\utc{\cA'}=\utc{\cA}\comp\Ba(\fF)$, we obtain
\[
\utc{\cA}\comp\fG_{\tc}=\tc=\fF\comp\tc'=\fF\comp\utc{\cA'}\comp\fG_{\tc'}=\utc{\cA}\comp\Ba(\fF)\comp\fG_{\tc'},
\]
which implies that $\fG_{\tc}=\Ba(\fF)\comp\fG_{\tc'}\colon\Ba(\aug{\cA_1})\otimes\Ba(\aug{\cA_2})\to\Ba(\cA)$. We claim that we can take $\aug{\fG}=\fG_{\tc}$. Indeed, it is clear by construction that $\fG_{\tc}$ is the identity on objects. Moreover, $\coBa(\fG_{\tc})=\coBa(\Ba(\fF))\comp\coBa(\fG_{\tc'})$ is a quasi-isomorphism because both $\coBa(\fG_{\tc'})$ and $\coBa(\Ba(\fF))$ are quasi-isomorphisms. To see this last fact, notice that $\fF$ is a quasi-isomorphism by \autoref{dgAnadj}, hence the same is true for $\coBa(\Ba(\fF))$.
\end{proof}

By definition the dg cofunctor $\fG\colon\red{\aug{\B(\cA_1)}\otimes\aug{\B(\cA_2)}}\to\B(\red{\aug{\cA_1}\otimes\aug{\cA_2}})$ of \autoref{weakequiv} is of the form $\fG=\Bi(\fN)$ for a unique $A_\infty$ bifunctor
\begin{equation}\label{mainbifun}
\fN\in\ACatn(\cA_1,\cA_2,\red{\aug{\cA_1}\otimes\aug{\cA_2}}).
\end{equation}

\begin{prop}\label{mainqeq}
Given $\cA_1,\cA_2\in\dgCatn$ and $\cA\in\dgCatc$, the $A_\infty$ functor
\[
\rcf{\fN}\colon\FunAn(\red{\aug{\cA_1}\otimes\aug{\cA_2}},\cA)\to\FunAn(\cA_1,\cA_2,\cA)
\]
is a quasi-equivalence.
\end{prop}

\begin{proof}
In order to prove that $H(\rcf{\fN})$ is essentially surjective we need to show that, for every $\fM\in\ACatn(\cA_1,\cA_2,\cA)$, there exists $\fF\in\ACatn(\red{\aug{\cA_1}\otimes\aug{\cA_2}},\cA)$ such that $\fM\hison\fF\comp\fN$. Setting $\cC:=\red{\aug{\B(\cA_1)}\otimes\aug{\B(\cA_2)}}$, by \autoref{multifundg} there exist unique $\dgmf{\fM}\in\dgCatn(\coB(\cC),\cA)$ and $\dgmf{\fN}\in\dgCatn(\coB(\cC),\red{\aug{\cA_1}\otimes\aug{\cA_2}})$ such that $\fM=\dgmf{\fM}\comp\ncbi[\cA_1,\cA_2]$ and $\fN=\dgmf{\fN}\comp\ncbi[\cA_1,\cA_2]$; moreover, $\dgmf{\fN}$ is a quasi-isomorphism because $\coB(\Bi(\fN))=\coB(\fG)$ is a quasi-isomorphism by \autoref{weakequiv}. Then there exists $\fH\in\ACatn(\red{\aug{\cA_1}\otimes\aug{\cA_2}},\coB(\cC))$ such that $\fH\comp\dgmf{\fN}\htpn\id_{\coB(\cC)}$, and we can take $\fF:=\dgmf{\fM}\comp\fH$. Indeed, we have $\fF\comp\dgmf{\fN}=\dgmf{\fM}\comp\fH\comp\dgmf{\fN}\htpn\dgmf{\fM}$, whence $\fF\comp\dgmf{\fN}\hison\dgmf{\fM}$. It follows that
\[
\fF\comp\fN=\fF\comp\dgmf{\fN}\comp\ncbi[\cA_1,\cA_2]\hison\dgmf{\fM}\comp\ncbi[\cA_1,\cA_2]=\fM,
\]
as required.

As for fully faithfulness of $H(\rcf{\fN})$, in view of \autoref{natcoder} we have to prove that
\[
\fG^*\colon\coDer(\Bi(\fF),\Bi(\fF'))\to\coDer(\Bi(\fF)\comp\fG,\Bi(\fF')\comp\fG)
\]
is a quasi-isomorphism, for every $\fF,\fF'\in\ACatn(\red{\aug{\cA_1}\otimes\aug{\cA_2}},\cA)$. By \autoref{dercoder} it is equivalent to show that
\[
\coB(\fG)^*\colon\Der(\dgmf{\fF},\dgmf{\fF'})\to\Der(\dgmf{\fF}\comp\coB(\fG),\dgmf{\fF'}\comp\coB(\fG))
\]
is a quasi-isomorphism. This last fact follows from \autoref{derqiso} below, taking into account that we can easily reduce to the case of algebras, and that $\coB(C)$ is cofibrant in $\dgAlgn$, for every $C\in\dgcoAlgn$.
\end{proof}

\begin{lem}\label{derqiso}
Let $f\colon A'\to A$ be a quasi-isomorphism in $\dgAlgn$ with $A',A$ cofibrant. Then $f^*\colon\Der(A,M)\to\Der(A',M)$ is a quasi-isomorphism, for every dg $A$-bimodule $M$. 
\end{lem}

\begin{proof}
We will use some results proved in \cite[Section 7]{H} in a more general operadic context.\footnote{Notice that, in the case of algebras over an associative operad, the notion of module must be interpreted as bimodule in the classical sense (see \cite[Proposition 12.3.1]{LV}). This explains why, throughout \cite[Section 7]{H}, modules (and not bimodules) are used.} First, for every $A\in\dgAlgn$, there exists a dg $A$-bimodule $\md_A$ representing the functor $\Der(A,\farg)$ (by \cite[7.2.2]{H})\footnote{In order to avoid confusion with cobar, we do not adopt the standard notation $\Omega_A$.} Moreover, every morphism $f\colon A'\to A$ in $\dgAlgn$ induces a morphism of dg $A$-bimodules $\md^f\colon\md_{A'}\otimes_{A'}A\to\md_A$ (by \cite[7.2.3]{H}), and it is easy to see that $f^*\colon\Der(A,M)\to\Der(A',M)$ can be identified with
\[
\dgMod{A}(\md^f,M)\colon\dgMod{A}(\md_A,M)\to\dgMod{A'}(\md_{A'},M)\iso\dgMod{A}(\md_{A'}\otimes_{A'}A,M),
\]
for every dg $A$ bimodule $M$. By \cite[7.3.3]{H}, the fact that $A',A$ are cofibrant in $\dgAlgn$ implies that $\md_{A'}$ is cofibrant in $\dgMod{A'}$ (whence $\md_{A'}\otimes_{A'}A$ is cofibrant in $\dgMod{A}$) and $\md_A$ is cofibrant in $\dgMod{A}$. Finally, by \cite[7.3.6]{H}, $\md^f$ is a quasi-isomorphism when $f$ is a quasi-isomorphism and $A',A$ are cofibrant. Thus, with our assumptions, $\md^f$ is a quasi-isomorphism between cofibrant dg $A$-bimodules, hence it is a homotopy equivalence. It follows that also $\dgMod{A}(\md^f,M)$ is a homotopy equivalence, and in particular a quasi-isomorphism.
\end{proof}

\subsection{Proof of \autoref{thm:internalHoms}}

Let $\cA_1$, $\cA_2$ and $\cA_3$ be dg categories. Consider the following diagram of natural bijections of sets:
\begin{equation}\label{eqn:quadratone}
\xymatrix{\Hqe(\cA_1\otimes\cA_2,\cA_3) \ar@{<->}[d]^-{1:1}_-{(\clubsuit)} & & \Hqe(\cA_1,\FunA(\cA_2,\cA_3)) \ar@{<->}[d]_-{1:1}^-{(\spadesuit)} \\
\QACatc(\cA_1\otimes\cA_2,\cA_3) \ar@{<->}[d]^-{1:1}_-{(\heartsuit)} & &  \ar@{<->}[d]_-{1:1}^-{(\blacklozenge)} \Hqe(\cA_1,\FunAc(\cA_2,\cA_3)) \\
\QACatc(\cA_1,\cA_2,\cA_3) \ar@{<->}[rr]^-{1:1}_-{(\bigstar)} & & \QACatc(\cA_1,\FunAc(\cA_2,\cA_3)),}
\end{equation}
where the existence of the bijections $(\clubsuit)$ and $(\blacklozenge)$ follow from \autoref{thm:localizations}, while the bijections $(\spadesuit)$ and $(\bigstar)$ are easy consequences of \autoref{lem:qefun} and of the following result, respectively.

\begin{prop}\label{prop:allqe}
For every $\cA_1,\cA_2,\cA\in\ACatc$ there is an isomorphism in $\ACatc$
\[
\FunAc(\cA_1,\cA_2,\cA)\iso\FunAc(\cA_1,\FunAc(\cA_2,\cA)).
\]
\end{prop}

\begin{proof}
It follows from \cite[Proposition 9.5]{BLM} together with \cite[Proposition 4.12]{BLM}.
\end{proof}

The rest of the section is devoted to proving that there is a natural bijection $(\heartsuit)$, since this would yield a natural bijection
\[
\xymatrix{
\Hqe(\cA_1\otimes\cA_2,\cA_3) \ar@{<->}[r]^-{1:1} & \Hqe(\cA_1,\FunA(\cA_2,\cA_3)),
}
\]
as claimed in \autoref{thm:internalHoms}.

\medskip

We start by defining two natural morphisms
\[
\finc\colon\cA_1\otimes\cA_2\to\red{\aug{\cA_1}\otimes\aug{\cA_2}}, \qquad \fpro\colon\red{\aug{\cA_1}\otimes\aug{\cA_2}}\to\cA_1\otimes\cA_2
\]
in $\dgCatn$. While $\finc$ is simply the natural inclusion, $\fQ$ is the restriction to $\red{\aug{\cA_1}\otimes\aug{\cA_2}}$ of $\fQ_1\otimes\fQ_2$, where $\fQ_i\colon\aug{\cA_i}\to\cA_i$ (for $i=1,2$) is the unique extension of $\id_{\cA_i}$ to a (unital) dg functor. It is obvious that $\fpro\comp\finc=\id_{\cA_1\otimes\cA_2}$. Although $\finc$ is really non-unital, we have the following easy result.

\begin{lem}\label{unittensor}
With the above notation, $\fpro$ is a (unital) dg functor and the $A_\infty$ multifunctor $\fpro\comp\fN\in\ACatn(\cA_1,\cA_2,\cA_1\otimes\cA_2)$ (where $\fN$ is the $A_\infty$ bifunctor defined in \eqref{mainbifun}) is cohomologically unital.
\end{lem}

\begin{proof}
A direct computation shows that $\id_{A_1}\otimes1_{A_2}-\id_{A_1}\otimes\id_{A_2}+1_{A_1}\otimes\id_{A_2}$ is the identity of $(A_1,A_2)$ in $\red{\aug{\cA_1}\otimes\aug{\cA_2}}$ (for every $A_1\in\cA_1$ and every $A_2\in\cA_2$), hence $\red{\aug{\cA_1}\otimes\aug{\cA_2}}\in\dgCat$. Moreover, the fact that $\fpro$ maps each of $\id_{A_1}\otimes1_{A_2}$, $\id_{A_1}\otimes\id_{A_2}$ and $1_{A_1}\otimes\id_{A_2}$ to $\id_{A_1}\otimes\id_{A_2}$, which is the identity of $(A_1,A_2)$ in $\cA_1\otimes\cA_2$, clearly implies that $\fpro$ is unital.

On the other hand, by definition $\fpro\comp\fN$ is cohomologically unital if and only if its restrictions $(\fpro\comp\fN)\rest{(\cA_1,A_2)}$ (for every $A_2\in\cA_2$) and $(\fpro\comp\fN)\rest{(A_1,\cA_2)}$ (for every $A_1\in\cA_1$) are cohomologically unital. Now, from the definition of $\fN$ it is easy to see that $\fN\rest{(\cA_1,A_2)}$ is the non-unital dg functor $\cA_1\to\red{\aug{\cA_1}\otimes\aug{\cA_2}}$ defined on objects by $A_1\mapsto(A_1,A_2)$ and on morphisms by $f\mapsto f\otimes1_{A_2}$. It follows that $(\fpro\comp\fN)\rest{(\cA_1,A_2)}=\fpro\comp\fN\rest{(\cA_1,A_2)}$ is the (unital) dg functor $\cA_1\to\cA_1\otimes\cA_2$ defined on objects by $A_1\mapsto(A_1,A_2)$ and on morphisms by $f\mapsto f\otimes\id_{A_2}$. Similarly for the restrictions $(\fpro\comp\fN)\rest{(A_1,\cA_2)}$.
\end{proof}

As a consequence, we obtain an $A_\infty$ functor
\[
\rcf{\fpro\comp\fN}\colon\FunAc(\cA_1\otimes\cA_2,\cA_3)\to\FunAc(\cA_1,\cA_2,\cA_3).
\]
Then $H^0(\rcf{\fpro\comp\fN})$ induces a natural map of sets
\[
\Psi\colon\QACatc(\cA_1\otimes\cA_2,\cA_3)\to\QACatc(\cA_1,\cA_2,\cA_3),
\]
sending the equivalence class of $\fF\in\ACatc(\cA_1\otimes\cA_2,\cA_3)$ to the equivalence class of $\fF\comp\fpro\comp\fN\in\ACatc(\cA_1,\cA_2,\cA_3)$. In order to conclude the proof of \autoref{thm:internalHoms}, it is therefore enough to prove the following result.

\begin{prop}\label{mainbij}
The map $\Psi$ is bijective.
\end{prop}

\begin{proof}
Given $\fF_1,\fF_2\in\ACatc(\cA_1\otimes\cA_2,\cA_3)$ such that $\fF_1\comp\fpro\comp\fN\hison\fF_2\comp\fpro\comp\fN$, it follows from \autoref{mainqeq} that $\fF_1\comp\fpro\hison\fF_2\comp\fpro$. Since $\fpro\comp\finc=\id_{\cA_1\otimes\cA_2}$, we obtain $\fF_1=\fF_1\comp\fpro\comp\finc\hison\fF_2\comp\fpro\comp\finc=\fF_2$, which proves that $\Psi$ is injective.

In order to prove that $\Psi$ is surjective we need to show that, for every $\fM\in\ACatc(\cA_1,\cA_2,\cA_3)$, there exists $\fF\in\ACatc(\cA_1\otimes\cA_2,\cA_3)$ such that $\fM\hison\fF\comp\fpro\comp\fN$. Again by \autoref{mainqeq}, there exists $\fF'\in\ACatn(\red{\aug{\cA_1}\otimes\aug{\cA_2}},\cA_3)$ such that $\fM\hison\fF'\comp\fN$, and we claim that $\fF'$ is actually cohomologically unital. To see this, notice that, since $\fF'\comp\fN$ is cohomologically unital, the same is true for its restrictions. Remembering the explicit description of the restrictions of $\fN$ given in the proof of \autoref{unittensor}, this implies that
\begin{equation}\label{id1}
H(\fF')(\id_{A_1}\otimes1_{A_2})=H(\fF')(1_{A_1}\otimes\id_{A_2})=\id_{\fF'(A_1,A_2)}
\end{equation}
for every $A_1\in\cA_1$ and every $A_2\in\cA_2$. As $\id_{A_1}\otimes\id_{A_2}=(\id_{A_1}\otimes1_{A_2})\comp(1_{A_1}\otimes\id_{A_2})$, we also have
\begin{equation}\label{idid}
H(\fF')(\id_{A_1}\otimes\id_{A_2})=\id_{\fF'(A_1,A_2)}.
\end{equation}
From \eqref{id1} and \eqref{idid} we immediately deduce that $H(\fF')$ is unital.

To conclude, it is then enough to prove that $\fF'\hison\fF'\comp\finc\comp\fpro$. Indeed, assuming this, $\fF:=\fF'\comp\finc\in\ACatn(\cA_1\otimes\cA_2,\cA_3)$ clearly satisfies 
\[
\fF\comp\fpro\comp\fN=\fF'\comp\finc\comp\fpro\comp\fN\hison\fF'\comp\fN\hison\fM.
\]
Moreover, $\fF$ is cohomologically unital because both $\fpro$ (by \autoref{unittensor}) and $\fF\comp\fpro\hison\fF'$ are cohomologically unital and $\fpro$ is surjective on objects.

Now let us prove that $\fF'\hison\fF'\comp\finc\comp\fpro$. It is straightforward to check that there is a (closed, degree $0$) dg natural transformation $T$ between the dg functors $\id_{\red{\aug{\cA_1}\otimes\aug{\cA_2}}}$ and $\finc\comp\fpro$ defined by $T_{(A_1,A_2)}:=\id_{A_1}\otimes\id_{A_2}$ for every $(A_1,A_2)\in\red{\aug{\cA_1}\otimes\aug{\cA_2}}$. By \cite[Section 1e]{Sei}, this yields a (closed, degree $0$) $A_\infty$ natural trasformation $\widetilde{T}:=\fF'\comp T$ between $\fF'$ and $\fF'\comp\finc\comp\fpro$. In view of \eqref{idid}, $H(\widetilde{T})_{(A_1,A_2)}=\id_{\fF'(A_1,A_2)}$ is an isomorphism in $H(\cA_3)$ for every  $(A_1,A_2)\in\red{\aug{\cA_1}\otimes\aug{\cA_2}}$. It is then easy to deduce from \cite[Lemma 1.6]{Sei} that $\widetilde{T}$ yields an isomorphism in $H(\FunAc(\red{\aug{\cA_1}\otimes\aug{\cA_2}},\cA_3))$, thus proving that $\fF'\hison\fF'\comp\finc\comp\fpro$.
\end{proof}


\bigskip

{\small\noindent{\bf Acknowledgements.} First of all, we would like to warmly thank Michel Van den Bergh for some insightful conversations and for providing a key observation which we used in the proof of \autoref{dgAadj}, and Bernhard Keller for some relevant comments on a first draft of this paper. We benefitted of many useful conversations with G.\ Faonte, F.\ Genovese, V.\ Melani, M.\ Porta, A.\ Rizzardo, G.\ Tabuada and B.\ Vallette. It is our pleasure to thank them. Part of this paper was written while the second author was visiting the Erwin Schr\"odinger International Institute for Mathematics and Physics and the third author was visiting Northeastern University. Their warm hospitality is gratefully acknowledged.}


\end{document}